\documentclass[12pt]{amsart}
\usepackage{amsfonts,amssymb, amscd, latexsym, graphicx, psfrag, color}
\usepackage[all]{xy}

\newtheorem{dummy}{dummy}[section]

\newtheorem{theorem}[dummy]{Theorem}

\newtheorem{corollary}[dummy]{Corollary}
\newtheorem{proposition}[dummy]{Proposition}
\theoremstyle{definition}
\newtheorem{definition}[dummy]{Definition}
\newtheorem{example}[dummy]{Example}
\newtheorem{remark}[dummy]{Remark}

\newcommand{\bC}{\mathbb{C}}

\newcommand{\bP}{\mathbb{P}}

\newcommand{\bR}{\mathbb{R}}
\newcommand{\bZ}{\mathbb{Z}}

\newcommand{\bfR}{\mathbf{R}}

\newcommand{\cB}{\mathcal{B}}
\newcommand{\cC}{\mathcal{C}}
\newcommand{\cE}{\mathcal{E}}

\newcommand{\cM}{\mathcal{M}}
\newcommand{\cO}{\mathcal{O}}
\newcommand{\cP}{\mathcal{P}}

\newcommand{\cR}{\mathcal{R}}

\newcommand{\cS}{\mathcal{S}}
\newcommand{\cT}{\mathcal{T}}
\newcommand{\cU}{\mathcal{U}}
\newcommand{\cV}{\mathcal{V}}
\newcommand{\cX}{\mathcal{X}}

\newcommand{\Spec}{\mathrm{Spec}\,}

\newcommand{\Fuk}{\mathrm{Fuk}}
\newcommand{\Sh}{\mathit{Sh}}
\newcommand{\Shv}{\mathrm{Shv}}
\newcommand{\Rep}{\mathit{Rep}}
\newcommand{\Hom}{\mathrm{Hom}}

\newcommand{\Perf}{\mathcal{P}\mathrm{erf}}

\newcommand{\MSh}{\mathit{MSh}}

\renewcommand{\SS}{\mathit{SS}}

\newcommand{\cpm}{\mathrm{CPM}}
\newcommand{\CPM}{\mathrm{CPM}}

\renewcommand{\SS}{\mathit{SS}}

\newcommand{\op}{\mathrm{op}}

\newcommand{\RGpc}{\mathcal{RG}_{\mathrm{pc}}}
\newcommand{\RGo}{\mathcal{RG}_{\mathrm{o}}}
\newcommand{\RGc}{\mathcal{RG}_{\mathrm{c}}}
\newcommand{\RTpc}{\mathcal{RT}_{\mathrm{pc}}}
\newcommand{\RTo}{\mathcal{RG}_{\mathrm{o}}}
\newcommand{\RTc}{\mathcal{RT}_{\mathrm{c}}}

\newcommand{\Cont}{\mathrm{Cont}}
\newcommand{\Fishn}{\mathrm{Chord}}

\newcommand{\dgCat}{\mathrm{dg}\mathcal{C}\mathrm{at}}
\newcommand{\dgCatS}{\dgCat^{\mathrm{S}}}
\newcommand{\SCat}{\mathrm{SCat}_\infty}

\newcommand{\Sp}{\mathrm{Sp}}
\newcommand{\SL}{\mathrm{SL}}

\newcommand{\Sch}{\mathrm{Sch}}
\newcommand{\St}{\mathrm{St}}
\newcommand{\Gpd}{\mathrm{Gpd}}

\newcommand{\Fun}{\mathrm{Fun}}

\begin{document}

\title[Ribbon Graphs and Mirror Symmetry I]{Ribbon Graphs and Mirror Symmetry I}

\begin{abstract}
Given a ribbon graph $\Gamma$ with some extra structure, we define,
using constructible sheaves, a
dg category $\cpm(\Gamma)$ meant to model the
Fukaya category of a Riemann surface in the cell
of Teichm\"uller space described by $\Gamma.$
When $\Gamma$ is appropriately decorated and admits
a combinatorial ``torus fibration with section,''
we construct from $\Gamma$
a one-dimensional algebraic stack $\widetilde{X}_\Gamma$ with toric components.
We prove that our model is equivalent to $\Perf(\widetilde{X}_\Gamma)$,
the dg category of
perfect complexes on $\widetilde{X}_\Gamma$.
\end{abstract}

\author{Nicol\'o Sibilla}
\address{Nicol\'o Sibilla, Department of Mathematics, Northwestern University,
2033 Sheridan Road, Evanston, IL  60208}
\email{sibilla@math.northwestern.edu}

\author{David Treumann}
\address{Eric Zaslow, Department of Mathematics, Northwestern University,
2033 Sheridan Road, Evanston, IL  60208}
\email{treumann@math.northwestern.edu}

\author{Eric Zaslow}
\address{Eric Zaslow, Department of Mathematics, Northwestern University,
2033 Sheridan Road, Evanston, IL  60208}
\email{zaslow@math.northwestern.edu}

\maketitle

{\small \tableofcontents}

\section{Introduction}

\subsection{Ribbon Graphs and HMS in One Dimension}
\label{rg}

Recall from the work of Harer, Mumford, Penner and Thurston (see, e.g., \cite{H}, \cite{P})
that ribbon graphs label cells in a
decomposition of the moduli
space of punctured Riemann surfaces.  The graph itself
is a retraction of the surface, also known as a \emph{skeleton} or \emph{spine.} 
Let $X_\Gamma$ denote a punctured Riemann surface
with spine $\Gamma.$
We imagine a Stein structure on $X_\Gamma$ so that $\Gamma$ is
the skeleton 
(the union of all stable manifolds)
of the Stein function. 
Kontsevich in \cite{K-stein} conjectured 
that the Fukaya category of a Stein manifold can be computed
locally on the skeleton and discussed applications to homological
mirror symmetry for Riemann surfaces.\footnote{Locality in the skeleton first 
appeared first in Kontsevich's
\cite{K-stein}, but is also part of a circle of ideas concerning the local
nature of the Fukaya category of exact symplectic manifolds, prevalent the work 
of Abouzaid \cite{A} and Seidel \cite{S-morse, S-spec} (see also \cite{AS}),
and in the relation of the Fukaya category to sheaf theory by Nadler \cite{N} and the last author \cite{NZ}.}
In this paper, we investigate this idea from the perspective
of constructible sheaves and T-duality.

We shall define, starting from a
ribbon graph $\Gamma$, a category $\cpm(\Gamma)$ (``constructible
plumbing model''), defined using
the language of
constructible sheaves.  $\cpm(\Gamma)$ serves as a stand-in for
the Fukaya category $Fuk(X_\Gamma)$ of the surface $X_\Gamma$.

When the ribbon graph is appropriately decorated and carries
a combinatorial version of a torus fibration with section,
then we can define a ``mirror'' curve $\widetilde{X}_\Gamma$,
an algebraic stack with toric components.  Exploiting
a Beilinson-Bondal-type equivalence called
the coherent-constructible correspondence \cite{B, FLTZ}, we prove
$$\cpm(\Gamma) \cong \Perf(\widetilde{X}_\Gamma),$$
an equivalence of dg categories.
We conjecture that $\cpm(\Gamma)\cong Fuk(X_\Gamma).$  Together
with the above result,
this would prove a one-dimensional version of homological
mirror symmetry.  
%(The conjectural equivalence between
%our model and the Fukaya category is the subject of work
%in progress with B. Fang -- see also \cite{A, K-stein}.)\xxx{remove remark?} 

\subsection{Basic Idea; Simple Example}
\label{sec:ex}

Recall from \cite{NZ,N} that the Fukaya category of a cotangent bundle $X = T^*Y$
is equivalent to constructible sheaves on the base, $Y.$  One can specialize
to sheaves that are constructible for a particular stratification $\cS$ by
only considering a subcategory of the Fukaya category generated
by Lagrangians which are asymptotic to a conical Lagrangian
submanifold $\Lambda_\cS\subset T^*Y$ depending on $\cS.$ 
The conical Lagrangian $\Lambda,$ encodes the
``microlocal behavior'' of the corresponding sheaf.  
When the conical Lagrangian submanifold $\Lambda$ is strictly $\bR_+$-invariant and
not necessarily $\bR$-invariant in the fibers, the corresponding subcategory of sheaves
has ``singular support'' in $\Lambda,$ a somewhat more refined microlocal condition than
that defined by a stratification.
For example, when $Y=S^1,$ $\Lambda$ will be a subset of the zero section together
with a finite number of rays of cotangent fibers -- i.e., 
a graph with valency $\leq 4$ at each vertex.
For instance, the category of sheaves on $S^1$ stratified by a point and
its complement corresponds to $\Lambda$ which is the union of the zero
section and the cotangent fiber of the point -- a graph with
one 4-valent vertex and one loop.  The graph is a ribbon graph, since
the symplectic geometry of the cotangent
of $S^1$ determines
a cyclic ordering of the edges.  It was proven in \cite{B, FLTZ} that constructible
sheaves on $S^1$ with this stratification are equivalent to
coherent sheaves on $\bP^1.$  Since this category of constructible
sheaves is equivalent to a Fukaya category on $T^*(S^1)$, by \cite{NZ,N}, this equivalence is a form of 
or mirror symmetry.

The model $\cpm$ is constructed by gluing pieces of a graph together,
and can be inferred from the following example.
Consider a family of conics $C_t$ in $\bP^2$
defined in homogeneous coordinates by $XY = t^2 Z^2$, or in inhomogeneous
coordinates by $xy = t^2.$  These conics are all isomorphic to $\bP^1$ when
$t\neq 0$ but degenerate to two copies of $\bP^1$ attached at a node
when $t=0.$  The conics all have an open orbit $C_t^\circ$ of a \emph{common} $\bC^*\supset S^1$
defined by $\lambda: (x,y)\mapsto (\lambda x, \lambda^{-1}y),$
i.e $C_t^\circ$ has an $S^1$ fibration over $\bR = \bC^*/S^1$.
We can pick a $({\bC}^*)^2$-equivariant line bundle on $\bP^2$ such as the
hyperplane bundle, give it a metric (and thus compatible connection)
invariant under $(S^1)^2 \supset S^1$,
restrict it to $C_t^\circ$ and compute the monodromies over the $\bR$ family
of $S^1$ orbits.
The resulting spectral curve, the ``T-dual Lagrangian'' $L_t,$ lies
in $T^*S^1$ (the dual $S^1$, actually),
for $t\neq 0.$  We study the degeneration $t\rightarrow 0.$

When $t\rightarrow 0$, $L_t$
``splits'' into two Lagrangians on two different cotangents
of two copies of a circle
stratified by a point $p$ and its complement -- but they are joined: $+\infty$
inside $T^*_pS^1$ from the perspective of one circle is glued to $-\infty$
of the other circle (just as the corresponding limits of
$\bC^*$ orbits on the two $\bP^1$ components are joined).  In terms of
coordinates $(x,y\sim y+1)$ on $T^*S^1,$ the $L_t$
look like $y = t^2\sinh(2x)/(1+t^2\cosh(2x)).$
This splitting is illustrated
for two values of $t$.

\begin{figure}[ht]
\begin{center}
%\psfrag{S}{$\Si$}
%\psfrag{B}{$\beta=\{(1,0),(0,1),(-1,-2)\}$}
\includegraphics[scale=0.65]{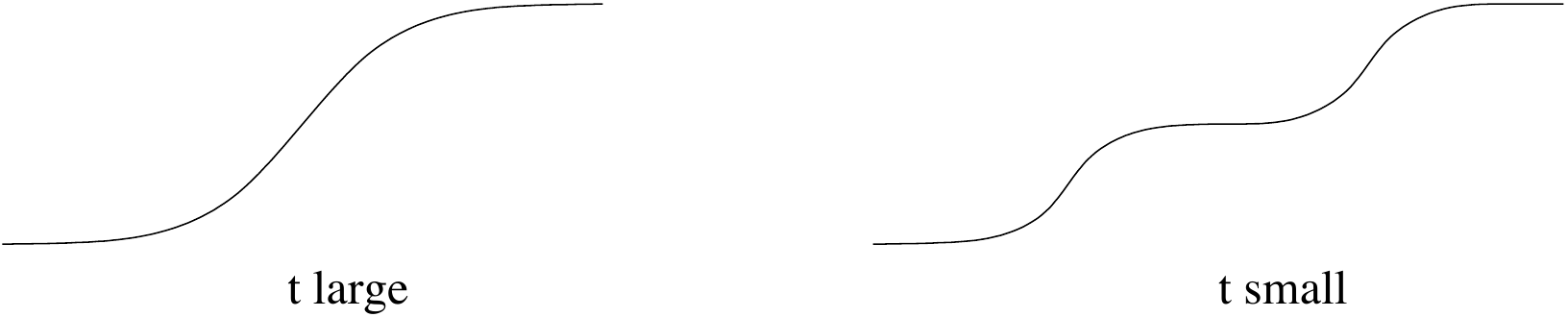}
\end{center}
\end{figure}

The picture suggests that
coherent sheaves (perfect complexes, actually \cite{FLTZ})
on the degenerate conic can be described by 
pairs of constructible sheaves on $S^1$ such that the microlocal stalk
of one sheaf along the ray to $+\infty$ in $T^*_pS^1$ equals the microlocal stalk of
the other sheaf along the ray to $-\infty$.  The graph $\Lambda$ for
such a gluing has two 4-valent vertices attached by a single
edge, and two loops -- a bit like a curtain rod with two rings. 
Extending this example, we can consider a further gluing of the
toric endpoints of the two $\bP^1$ components.  It becomes clear
what to do to build a constructible model for this algebraic
curve:  glue the curtain rod along its two ends.  The graph is a
circle with two circles attached at two distinct points -- a combinatorial
version of a torus fibration!

The algebraic curve obtained in this way is a degenerate elliptic curve --
a Calabi-Yau manifold with toric components, such as what appears at
the large complex structure limit point in mirror symmetry.
The ribbon graph category is a model for the Fukaya category of the
exact symplectic manifold at the large radius limit (the
symplectic manifold gets compactified in
deformations \cite{S-quartic} which are mirror to smoothing the
elliptic curve) corresponding to the ribbon graph,
and is equivalent to coherent sheaves on the algebraic mirror
Calabi-Yau.  Through this reasoning, then, we get a model for homological
mirror symmetry of degenerate Calabi-Yau manifolds.

To summarize:
\begin{itemize}
 \item Each four-valent vertex is modeled by the category
$Sh(\bR,\cdot)$ of constructible
sheaves on $\bR$ stratified by a point and its complement.
\item Two opposing half-edges of a four-valent vertex
represent the two $\bR$ directions,
while the remaining two represent the \emph{microlocal stalks} of the
constructible sheaves.  More succinctly,
$Sh(\bR,\cdot)$ has four functors to chain complexes (constructible
sheaves on edge intervals) defined
by the stalks or microlocal stalks at the rays corresponding
to the four half-edges.
\item Categories are glued as ``fiber products,'' i.e. by
taking objects at each vertex category and requiring that the
images of objects at two vertices agree along an edge which joins them.
\item When the graph has a map to a cycle graph, with circle
fibers over the vertices, we can join together the toric curves
corresponding to the fibers to make
a degenerate Calabi-Yau.  Coherent sheaves on this
Calabi-Yau will be isomorphic to the constructible plumbing
model of the ribbon
graph.
\end{itemize}

The discussion thus far has been quite informal.  The paper is dedicated to
making these observations rigorous.  In Section 2,
we introduce the notions needed to take fiber products and more general limits of dg categories,
and review material on algebraic stacks and ribbon graphs.
In Section 3, we define \emph{chordal} ribbon graphs, which encode the structure of
being ``locally like $T^*\bR$'' used to relate to constructible sheaves.  We
review the microlocal theory of sheaves in one dimension and define $\cpm(\Gamma)$
from a chordal ribbon graph $\Gamma$ by gluing, as discussed above.
In Section 4, we define the notion of ``dualizable'' ribbon graph, $\Gamma$ --
the structure called a ``combinatorial version of a torus fibration'' above --
which allows us to construct an algebraic curve $\widetilde{X}_\Gamma$ with toric
components such that $\Perf(\widetilde{X}_\Gamma)\cong \cpm(\Gamma).$

Our discussion of Fukaya categories and skeleta
suggests that the proper definition of $\cpm$ should not depend on its chordal structure.
Instead, one should work with the more general notion of a \emph{graded ribbon graph}.
In Section 5, we introduce the category of graded ribbon graphs, then
define a \emph{Fukaya formalism} to be a functor from this
category to dg categories, such that contractions are
represented by isomorphisms.  Proving that $\cpm$ extends to a Fukaya formalism is the
subject of work in progress.

\subsection{Influences}
\label{sec:influences}
Our construction is influenced by the work of many others,
 and we
are grateful for their inspiring works.  We highlight a few.
\begin{itemize}
\item \emph{Abouzaid} has a plumbing construction \cite{A} which
is equivalent to our own, in certain circumstances,
and has several other parallel constructions, e.g. \cite{Ab} (see also \cite{AS}).
 \item \emph{Beilinson-Bondal} describe combinatorial models for
coherent sheaves on toric varieties.  Bondal's reinterpretation \cite{B}
of Beilinson's quiver in terms of constructible sheaves was
an inspiration behind the coherent-constructible correspondence
defined in \cite{FLTZ}.
\item \emph{Kontsevich} conjectured \cite{K-stein} that the
Fukaya category is local in the skeleton.  The rumor of this announcement
inspired
our search for a model using the microlocal structure of sheaf theory.
\item \emph{Nadler} has pushed the envelope in relating
Fukaya categories to algebraic topology \cite{N} (we note also \cite{NT}).
\item \emph{Seidel} initiated the study of exact symplectic
manifolds as mirrors of Calabi-Yau manifolds at their toric
degeneration point (large complex structure limit), especially
in his {\sl tour de force} work \cite{S-quartic}.  Seidel discusses the
locality of the Fukaya category in \cite{S-morse} and (among other things) applies these
ideas to Landau-Ginzburg theories in \cite{S-spec}.
\end{itemize}

\emph{Acknowledgments:}  We are grateful to Kevin Costello, Dima Tamarkin,
and David Nadler for their interest and help.  We are greatly indebted
to Bohan Fang and Chiu-Chu Melissa Liu for sharing their thoughts, and
(for the second- and third-named authors) for our
several collaborations with them.
The work of EZ is supported in part by NSF/DMS-0707064.

\section{Notation and conventions}

\subsection{Background and categories and sheaves}

In this section we review material on category theory and constructible sheaves.  

As we do not make any direct comparisons to the ``genuine'' Fukaya category, we do not work with $A_\infty$-categories; instead, we work with dg categories.  Since the plumbing perspective we take in this paper requires us to build new dg categories out of old, we have to have a good handle on ``what dg categories form.''  The answer is contained in the work of many people, including Drinfel'd, Tabuada, To\"en, and Lurie.  We may summarize by saying: dg categories form a Quillen model category, which we may regard as a ``$\infty$-category'' or ``quasi-category'' via a simplicial nerve construction.  In section \ref{sec:inftyanddg} we recall what we will need from this theory.

We work over the ground field $\bC$.

\subsubsection{$\infty$-and dg categories}
\label{sec:inftyanddg}

We refer to \cite{htt} for the theory of $\infty$-categories, and use the same notation.  An $\infty$-category is a simplicial set satisfying the weak Kan condition (\cite[Definition 1.1.2.4]{htt}).  Common sources of $\infty$-categories are:
\begin{enumerate}
\item If $\cC$ is an ordinary category, then the nerve $N\cC$ is an $\infty$-category. (\cite[Proposition 1.1.2.2]{htt}).

\item If $\cC$ is a simplicial category, i.e. a category enriched in simplicial sets, with the property that the simplicial set of maps between any two objects of $\cC$ is a Kan complex, then the simplicial nerve $N\cC$ is an $\infty$-category.  (\cite[Proposition 1.1.5.10]{htt}).

\item In particular if $\cC$ is a simplicial model category and $\cC_{cf}$ is the full subcategory spanned by objects that are both fibrant and cofibrant, then $N\cC_{cf}$ is an $\infty$-category, sometimes called the ``underlying $\infty$-category'' of $\cC$.

\item
If $K$ and $\cC$ are simplicial sets we write $\Fun(K,\cC)$ for the simplicial set given by $\Delta^n \mapsto \Hom(K \times \Delta^n,\cC)$.  If $\cC$ is an $\infty$-category then so is $\Fun(K,\cC)$ (\cite[Proposition 1.2.7.3]{htt}). 

%\item Fibrant simplicial subsets are regarded as ``$\infty$-groupoids,'' i.e. $\infty$-categories each of whose morphisms is invertible.  If $\cC$ is an $\infty$-category, then there is a fibrant simplicial subset $\cC_0$ obtained by discarding the noninvertible arrows.  In particular we write $\Fun_0(K,\cC)$ for the $\infty$-groupoid obtained from $\Fun(K,\cC)$ by discarding noninvertible natural transformations.

\end{enumerate}

The category of $\bC$-linear dg categories has a Quillen model structure introduced by Tabuada \cite{Tabuada} and studied by To\"en \cite{Toen}.  It has a natural simplicial enrichment \cite[Section 5]{Toen}.  We write $\dgCat$ for the underlying $\infty$-category of this simplicial model category.  We write $\dgCatS \subset \dgCat$ for the full subcategory (that is, sub-$\infty$-category) whose objects are triangulated dg categories.

\begin{remark}
\label{rem:dgcatsarecats}
We may realize $\dgCatS$ as a subcategory of the $\infty$-category $\SCat$ of \emph{stable $\infty$-categories} \cite{dag1}, and as a full subcategory of the $\infty$-category $\SCat^\bC$ of $\bC$-linear stable $\infty$-categories.  The latter category is defined as a full subcategory of the category of module categories for the symmetric monoidal $\infty$-category $\bC\text{-mod}$ \cite{dag3}.
\end{remark}

We have a notion of limits and colimits in $\infty$-categories, that behave very much like the classical 1-categorical notions \cite[Chapter 4]{htt}.  As $\dgCat$ is the underlying $\infty$-category of a simplicial model category, it admits small limits and colimits \cite[Corollary 4.2.4.8]{htt}.

%\begin{remark}
%Let us pay particular attention to the notion of an equalizer in $\dgCat$.  Suppose we have two dg categories $\cC_1$ and $\cC_2$ and two functors $F,G:\cC_1 \to \cC_2$.  The equalizer is $ \varprojlim (\cC_1 \righrightarrow \cC_2)$.  As it is determined by an $\infty$-categorical universal property this limit is only well-defined up to a ``contractible ambiguity.''  
%\end{remark}

\subsubsection{Grothendieck topologies and sheaves of dg categories}

We will need to describe Grothendieck topologies using the language of sieves.  Let $\cX$ be a small category, and let $X \in \cX$ be an object.  A \emph{sieve} on $X$ is a collection of arrows $Y \to X$ with the property that if $Y \to X$ belongs to the sieve and there exists a morphism $Y_1 \to Y$ then the composite $Y_1 \to Y \to X$ also belongs to $X$.

\begin{example}
If $\cX$ is the partially ordered set of open subsets of a topological space $X$, and $\{U_i\}$ is an open cover of $U \subset X$, then the collection of all open inclusions $V \hookrightarrow U$ that factor through one of the $U_i$ is a sieve on $U$.
\end{example}

We may regard a sieve as a full subcategory $\cU \subset \cX_{/X}$ of the comma category $\cX_{/X}$.  Given a sieve $\cU \subset \cX_{/X}$, and a morphism $f:Y \to X$, we define a sieve $f^* \cU$ on $Y$ by putting $U \to Y$ in $f^* \cU$ if the composite $U \to Y \to X$ belongs to $\cU$.

\begin{definition}
\label{def:gt}
A \emph{Grothendieck topology} on $\cX$ consists of, for each object $X$ of $\cX$, a collection of sieves on $X$ called ``covering sieves,'' subject to the following conditions:
\begin{itemize}
\item If $X$ is an object of $\cX$, then the trivial subcategory $\cX_{/X} \subset \cX_{/X}$ is a covering sieve.
\item If $f:Y \to X$ is a morphism and $\cU$ is a covering sieve on $X$, then $f^* \cU$ is a covering sieve on $Y$
\item Let $X$ be an object of $\cX$, let $\cU$ be a covering sieve on $X$, and let $\cV$ be an arbitrary sieve on $X$.  Suppose that, for each $f:Y \to X$ belonging to $\cU$, the pullback $f^* \cV$ is a covering sieve on $Y$.  Then $\cV$ is a covering sieve on $X$.
\end{itemize}
\end{definition}

In classical category theory, there is an order-preserving bijection between sieves on $X$ and subobjects of the functor $\Hom(-,X)$ from $\cX^{\op}$ to the category of sets: the subset of $\Hom(Y,X)$ corresponding to a sieve $U$ is the set of all morphisms $Y \to X$ belonging to $\cU$.  This construction extends to the $\infty$-categorical setting \cite[Proposition 6.2.2.5]{htt}: in particular, if $j:\cX \to \Fun(\cX^{\op},\cS)$ denotes the $\infty$-categorical Yoneda embedding, then for each sieve $\cU$ on $X$ we may associate an object $j(\cU) \in \Fun(\cX^{\op},\cS)$ and a morphism $j(\cU) \to j(X)$.

\begin{definition}
If $\cX$ is endowed with a Grothendieck topology, then a functor $F:\cX^{\op} \to \cS$ is a \emph{sheaf} if $\Hom(j(X),F) \to \Hom(j(\cU),F)$ is a weak homotopy equivalence for every $F$.
\end{definition}

\begin{remark}
The $\infty$-categorical Yoneda lemma \cite[Prop 5.1.3.1]{htt} identifies $\Hom(j(X),F)$ with $F(X)$.  We may furthermore identify $\Hom(j(\cU),F)$ with the inverse limit $\varprojlim_{(U \to X) \in \cU} F(U)$.  The sheaf condition is equivalent to the statement that the natural map
$$F(X) \to \varprojlim_{(U \to X) \in \cU} F(U)$$
is an equivalence for every covering sieve $\cU$---i.e. that ``$F(X)$ can be computed locally.''
\end{remark}

Write $\Shv(\cX) \subset \Fun(\cX^{\op},\cS)$ for the full subcategory of sheaves.  Let $\cC$ be an $\infty$-category.  A \emph{$\cC$-valued sheaf} on $\cX$ is a functor $\Shv(\cX)^{\op} \to \cC$ that converts small colimits in $\Shv(\cX)^{\op}$ to limits in $\cC$.  If $\cC$ has small limits then this is equivalent to a contravariant functor $\cX^{\op} \to \cC$ with the property that the natural map
$$F(X) \to \varprojlim_{(U \to X) \in \cU} F(U)$$
is an equivalence whenever $\cU$ is a covering sieve.  Write $\Shv(\cX,\cC)$ for the $\infty$-category of $\cC$-valued sheaves on $\cX$.

\begin{remark}
\label{rem:sheafofsheaves}
The construction $\Shv(\cX,\cC)$ provides a natural $\infty$-categorical version of some standard constructions in sheaf theory.  Let $X$ be a topological space and let $\cX$ be the Grothendieck site of open subsets of $X$.  We will write $\Shv(X,\cC)$ instead of $\Shv(\cX,\cC)$.  
\begin{enumerate}
\item  Let $\bC\text{-mod}$ be the $\infty$-category associated to the Quillen model category of cochain complex of complex vector spaces.  Then $\Shv(X,\bC\text{-mod})$ is an $\infty$-category whose homotopy category is naturally identified with $D(X;\bC)$, the unbounded derived category of sheaves of vector spaces on $X$.
\item $\Shv(X,\bC\text{-mod})$ is a $\bC$-linear stable $\infty$-category.  By remark \ref{rem:dgcatsarecats}, we may therefore identify it with a triangulated dg category via the equivalence $\SCat^\bC \cong \dgCat$.  The assignment $U \mapsto \Shv(U,\bC\text{-mod})$, where $U$ runs through open subsets of $X$, together with the restriction functors associated to inclusions $V \subset U$, assembles to an object of $\Shv(X,\dgCat)$---the ``sheaf of sheaves'' on $X$.
\end{enumerate}
\end{remark}

A sheaf on a Grothendieck site is determined by its behavior on a ``basis'':

\begin{proposition}
\label{prop:base}
Let $\cX$ be a Grothendieck site and let $\cB \subset \cX$ be a full subcategory with the property that each object of $\cX$ admits a covering sieve $\{U_i \to X\}$ such that each $U_i$ is in $\cB$.  Then the restriction functor $\Shv(\cX,\cC) \to \Shv(\cB,\cC)$ is an equivalence of $\infty$-categories.
\end{proposition}

\subsection{Background on algebraic stacks}

We will consider Deligne-Mumford stacks (that is, stacks with finite isotropy groups) defined over the complex numbers.

\begin{definition}
Let $\Gpd$ denote the full subcategory of the $\infty$-category of spaces spanned by spaces that can be obtained as the nerve of a groupoid.  Let $\Sch_{/\bC}$ denote the category of complex algebraic schemes.
A \emph{stack} is a functor $X:\Sch_{/\bC}^{\op} \to \Gpd$ that is a sheaf in the \'etale topology on $\Sch_{/\bC}$.
\end{definition}

We may regard the functor represented by an object of $S$ as valued in $\Gpd$, by regarding each $\Hom(T,S)$ as a discrete groupoid, in which case it becomes a stack.  We will abuse notation and denote this representable stack by $S$.  We say that a morphism of stacks $X \to Y$ is \emph{representable} if for each map $S \to Y$ with $S$ representable, the fiber product $X \times_Y S$ is also representable.  If P is a property of morphisms between schemes, then we will say that a representable morphism $X \to Y$ has property P if all base-changed maps $X \times_Y S \to S$ with $S$ representable have property P.

We restrict our attention to Deligne-Mumford stacks:

\begin{definition}
A stack $X:\Sch_{/\bC}^{\op} \to \Gpd$ is \emph{Deligne-Mumford} if it satisfies the following conditions:
\begin{enumerate}
\item There exists a representable \'etale morphism $U \to X$ where $U$ is itself representable
\item The diagonal map $X \to X \times X$ is representable and finite.
\end{enumerate}
We let $\St_{/\bC}$ denote the full subcategory of the $\infty$-category of functors $\Sch^{\op} \to \Gpd$ spanned by Deligne-Mumford stacks.
\end{definition}

\begin{remark}
The $\infty$-category $\St_{/\bC}$ can be obtained from the usual 2-category of Deligne-Mumford stacks via a 2-categorical nerve construction.
\end{remark}
 
A complex of quasicoherent sheaves on a scheme $X$ is \emph{perfect} if it is locally quasi-isomorphic to a bounded complex of vector bundles.  Perfect complexes are preserved by pullback.  We have descent for perfect complexes:

\begin{proposition}
\label{prop:descent}
There is a functor $\Perf:\Sch_{/\bC}^{\op} \to \dgCatS$ whose value on a scheme $X$ is the triangulated dg category of perfect complexes on $X$, and whose value on a morphism $f:X \to Y$ is the derived pullback functor $f^*$.  This functor is a sheaf in the \'etale topology on $\Sch_{/\bC}$.
\end{proposition}

The \'etale topology on $\Sch_{/\bC}$ forms a basis for a topology on $\St_{/\bC}$ which we also call the \'etale topology.  Proposition \ref{prop:base} says that the assignment $X \mapsto \Perf(X)$ makes sense for stacks as well as schemes.  When $X$ is a stack objects of $\Perf(X)$ will be called ``perfect complexes'' on $X$.

\subsection{Background on graphs and ribbon graphs}
\label{sec:pogarg}

\subsubsection{Graphs}

\begin{definition}
\begin{enumerate}
\item
A \emph{topological graph} is a tuple $(X,V_X)$ where $X$ is a locally compact topological space, $V_X \subset X$ is a finite closed subset, and each connected component of the open set $X - V_X$ is homeomorphic to an open interval.  We call the components of $X - V_X$ \emph{edges} of the topological graph.  An edge is \emph{compact} if its closure is compact, and \emph{noncompact} otherwise.
\item A \emph{graph} is a topological graph together with an open embedding $x_e:e \hookrightarrow \bR$ for each edge $e$ of $X$.  We require that the image of each $x_e$ is bounded.
\end{enumerate}
\end{definition}

An edge is a \emph{loop} if it is compact and its closure contains only one vertex, or equivalently if its closure is homeomorphic to a circle.  For simplicity we will assume from now on that our graphs have no loops, but we will allow noncompact and ``multiple'' edges---that is, we allow more than one edge to be incident with the same pair of vertices.

\begin{definition}
Let $(X,V_X)$ be a graph.  A \emph{half-edge} of $X$ incident with a vertex $v \in V_X$ is the germ of a connected component of a deleted neighborhood of $v$ in $X$.  The \emph{degree} of a vertex is the number of half-edges incident with $v$, and a graph is called \emph{locally finite} if each vertex has finite degree.
\end{definition}

If $\epsilon$ is a half-edge of $X$, we write $s(\epsilon)$ for the vertex incident with $\epsilon$ and $t(\epsilon)$ for the edge that contains $\epsilon$.

\begin{definition}
Let $X,V_X$ be a locally finite graph.
\begin{enumerate}

\item A \emph{compact closed walk} in a graph $(X,V_X)$ is a cyclically ordered sequence of edges
$$e_1,\ldots,e_r,e_{r+1} = e_1$$
with the property that for each $i$ the edges $e_i$ and $e_{i+1}$ are incident with the same vertex.  We regard the compact closed walks $(e_1,\ldots,e_r)$ and $(e_2,\ldots,e_r,e_1)$ as equivalent.
\item A \emph{noncompact closed walk} is a totally ordered sequence of edges
$$e_1,\ldots,e_r$$
such that $e_i$ and $e_{i+1}$ are incident with the same vertex, and such that $e_1$ and $e_r$ are noncompact edges of $X$.
\end{enumerate}
A closed walk is \emph{simple} if no vertex is incident with more than two edges belonging to the walk.
\end{definition}

\begin{definition}
Let $(X,V_X)$ and $(Y,V_Y)$ be topological graphs.  A \emph{morphism} $(X,V_X) \to (Y,V_Y)$ is a continuous map $u:X \to Y$ that carries vertices to vertices, that collapses some edges to vertices, and that otherwise maps each remaining edge homeomorphically onto an edge of $Y$.  That is, $u$ satisfies
\begin{itemize}
\item $u(V_X) \subset V_Y$
\item $u^{-1}(V_Y)$ is a union of vertices and edges of $X$
\item $u$ restricted to any edge in $X - u^{-1}(V_Y)$ is a homeomorphism onto an edge of $Y$.
\end{itemize}
If $(X,V_X)$ and $(Y,V_Y)$ are graphs, i.e. topological graphs endowed with coordinates $x_e:e \to \bR$ on their edges, then we say that a morphism $u:X \to Y$ of topological graphs is a morphism of graphs if for each edge $e \subset X$ which maps homeomorphically onto an edge $f \subset Y$, it is of the form $x_f = a x_e + b$ for $a,b \in \bR$.  That is, a morphism of graphs is given by an affine transformation on each edge.
\end{definition}

\begin{remark}
Though our model for them is topological, the theory of graphs and their morphisms as we have defined them is essentially combinatorial.  For instance, every topological graph admits a graph structure, and and two graphs with the same underlying topological graph are isomorphic to each other by a unique isomorphism which fixes the vertices and is homotopic to the identity relative to the vertices.  We will usually abuse notation and suppress the coordinate functions $\{x_e\}$ from the data of a graph $(X,V_X,\{x_e\})$.
\end{remark}

\subsubsection{Cyclically ordered sets}

\begin{definition}
Let $C$ be a finite set.  A \emph{cyclic order} on $C$ is a ternary relation $\cR_C \subset C \times C \times C$ satisfying the following axioms:
\begin{enumerate}
\item If $(x,y,z) \in \cR$ then $(y,z,x) \in \cR$
\item No triple of the form $(x,y,y)$ belongs to $\cR$
\item If $x,y,z$ are all distinct, then $\cR$ contains exactly one of the triples $(x,y,z)$ or $(x,z,y)$.
\item If $(x,y,z) \in \cR$ and $(y,z,w) \in \cR$, then $(x,y,w) \in \cR$ and $(x,z,w) \in \cR$.
\end{enumerate}
\end{definition}

\begin{remark}
Informally, $(x,y,z)$ belongs to $\cR$ if ``when traveling counterclockwise around $C$ starting at $x$, one encounters $y$ before $z$.''
\end{remark}

\begin{remark}
A subset of a cyclically ordered set $C' \subset C$ has a natural cyclic order itself: we set $\cR' = \cR \cap (C' \times C' \times C')$.  We call this the \emph{induced cyclic order} on $C'$.
\end{remark}

\begin{remark}
If $C$ is a cyclically ordered set and $c \in C$, then $C - \{c\}$ has a total ordering defined by $a < b$ if $(c,a,b) \in \cR$.  Write $R(c)$ for the minimal element of the ordered set $C - \{c\}$.  If $C$ has $n$ elements, then the map $c \mapsto R(c)$ has $R^n(c) = c$ and gives $C$ the structure of a $\bZ/n$-torsor.
\end{remark}

\begin{remark}
It will be useful to have the following language.  Let $C$ be a cyclically ordered set.  Then an element of $C \times C$ of the form $(c,R(c))$ is called a \emph{minimal pair}.  A $k$-tuple of the form $(c,Rc,R^2 c,\ldots, R^k c)$ is called a \emph{minimal string} or \emph{minimal $k$-tuple}.  Note that a cyclic order is determined by its set of minimal pairs.
\end{remark}

\begin{definition}
Let $C$ be a cyclically ordered set, and let $C_1$ and $C_2$ be two subsets of $C$.  We say that $C_1$ and $C_2$ are \emph{noninterlacing} if any four distinct elements $x_1,y_1,x_2,y_2$ with $x_1,y_1 \in C_1$ and $x_2,y_2 \in C_2$ satisfy the following condition:
\begin{itemize}
\item either $(x_1,y_1)$ or $(y_1,x_1)$ is a minimal pair of $\{x_1,y_1,x_2,y_2\}$ in its induced cyclic order.
\end{itemize}
\end{definition}

\begin{definition}
\label{def:join}
If $C_1$ and $C_2$ are cyclically ordered sets, and $p \in C_1$, $q \in C_2$ are elements, then there is a unique cyclic order on the set $C_{pq} := (C_1 -\{p\}) \coprod (C_2 - \{q\})$ in which the sets $C_1 - \{p\}$ and $C_2 - \{q\}$ are noninterlacing.  We refer to $C_{pq}$ as the \emph{join} of $C_1$ and $C_2$ along $p$ and $q$.
\end{definition}

%The join construction may be pictured as follows:
%\begin{center}
%\includegraphics{joinfirst.pdf}
%\end{center}

\subsubsection{Ribbon graphs}

\begin{definition}
Let $(X,V_X)$ be a graph in which every vertex has degree $\geq 2$.  A \emph{ribbon structure} on $(X,V_X)$ is a collection $\{\cR_v\}_{v \in V}$ where $\cR_v$ is a cyclic order on the set of half-edges incident with $v$.  We call a graph equipped with a ribbon structure a ribbon graph.
\end{definition}

\begin{remark}
A graph without vertices is necessarily homeomorphic to a disjoint union of open intervals.  Such a graph has a unique ribbon structure---the cyclic orders $\{\cR_v\}_{v \in V}$ are indexed by the empty set.
\end{remark}

\begin{remark}
\label{rem:embgraph}
Let $X$ be a graph and let $X \hookrightarrow W$ be an embedding into a surface.  An orientation on $W$ determines a ribbon structure on $X$.
\end{remark}

\begin{definition}
Let $(X,V_X,\{\cR_v\})$ be a ribbon graph with no isolated edges.  A \emph{compact (resp. noncompact) boundary component} of $X$ is a simple compact (resp. noncompact) closed walk with the following property: if $e$ and $f$ are consecutive edges in the walk with common vertex $v$, and associated half-edges $\overline{e}$ and $\overline{f}$ incident with $v$, the $(\overline{e},\overline{f})$ is a minimal pair in the cyclical order $\cR_v$.
\end{definition}

\begin{remark}
If $X$ does have isolated edges, then we include in the set of noncompact boundary components pairs $(e,\mathfrak{o})$ where $e$ is an isolated edge and $\mathfrak{o}$ is an orientation of $e$.
\end{remark}

\begin{remark}
If a closed walk is a boundary component, its reversal is not always a boundary component.  For instance, the reversal of a noncompact boundary component is never a boundary component.  In case a closed walk and its reversal are both boundary components, we regard them as different boundary components.
\end{remark}

\begin{remark}
If $X$ is a compact ribbon graph, write $v(X)$ for the number of vertices, $e(X)$ for the number of edges, $b(X)$ for the number of (necessarily compact) boundary components, and $g(X)$ for the unique integer satisfying
$$v(X) - e(X) + b(X) = 2  - 2g(X)$$
We call $g(X)$ the \emph{genus} of $X$.  Note that we do not define the genus of a noncompact ribbon graph.
\end{remark}

Let $X$ be a \emph{tree}---that is, an acyclic graph---endowed with a ribbon structure.  Since we require each vertex of a ribbon graph to have degree $\geq 2$, a ribbon tree is necessarily noncompact.  We will refer to the noncompact edges of $X$ as \emph{leaves} of the tree, and the remaining edges \emph{internal}.  Contracting an internal edge of the tree
A generalization of the join construction gives us a cyclic order on the set of leaves of any ribbon tree:

\begin{proposition}
\label{prop:jointree}
Let $X$ be a ribbon tree with $n$ leaves, and let $L$ be the set of leaves of $X$.  Define an operator $R:L \to L$ by setting $Re = f$ if there is a noncompact boundary component which begins at $e$ and ends at $f$.  Then $R$ gives $L$ the structure of a $\bZ/n$-torsor.  In particular, the leaves of any ribbon tree are naturally endowed with a cyclic order.
\end{proposition}

\section{CPM of chordal ribbon graphs}

\subsection{Microlocal sheaf theory in one dimension}

In this section we review some of the constructions of microlocal sheaf theory in the case where the base manifold is one dimensional.  

Let $M$ be a one-dimensional manifold (with or without affine structure).  By a \emph{sheaf} we will mean a cochain complex of sheaves of $\bC$-vector spaces.  Such a sheaf is \emph{constructible} if it satisfies the following two conditions:
\begin{enumerate}
\item Each stalk is quasi-isomorphic to a bounded complex of vector spaces.
\item The sheaf is cohomologically locally constant away from a discrete closed subset of $M$.
\end{enumerate}

Constructible sheaves on $M$ form a dg category that we will denote by $\Sh_c(M)$ or usually just by $\Sh(M)$.  Let $\cM$ be the category whose objects are 1-manifolds and whose morphisms are open immersions.  The assignments $M \mapsto \Sh(M)$ together with pullback along open immersions define a sheaf of dg categories on $\cM$---a subsheaf of the ``sheaf of sheaves'' of remark \ref{rem:sheafofsheaves}.  

We can simplify the development of the microlocal theory somewhat by endowing our one-manifolds with affine structures.

\begin{definition}
Let $M$ be a 1-manifold.  An \emph{affine structure} on $M$ is the data of an identification $\psi_U:U \cong (a,b) \subset \bR$ for every sufficiently small connected open subset of $M$, such that the transition maps $\psi_U \circ \psi_V^{-1}$ are of the form $cx + d$ with $c \neq 0$.
\end{definition}

If $M$ and $N$ are affine 1-manifolds then we have evident notions of affine morphisms $M \to N$.  In particular we may speak of affine $\bR$-valued functions and their germs on $M$.  The map $f \mapsto df$ identifies the cotangent bundle $T^*M$ of $M$ with the space of pairs $(x,\xi)$ where $x$ is a point of $M$ and $\xi$ is the germ of an affine $\bR$-valued function with $\xi(x) = 0$.

\begin{definition}[Morse groups/microlocal stalks]
Let $M$ be an affine 1-manifold, let $x$ be a point of $M$ a let $f$ be the germ of an affine $\bR$-valued function on $M$ around $x$.  For $\epsilon > 0$ sufficiently small let $A$ be the sublevel set $\{y \in M \mid y < f(x) + \epsilon\}$ and let $B$ be the sublevel set $\{y \in M \mid y < f(x) - \epsilon\}$.  We define a functor $\mu_{x,f}:\Sh(M) \to \bC\text{-mod}$ to be the cone on the natural map
$$\Gamma(A;F\vert_A) \to \Gamma(B;F\vert_B)$$
Since every constructible sheaf $F$ is locally constant in a deleted neighborhood of $x$, this functor does not depend on $\epsilon$ as long as it is sufficiently small.
\end{definition}

Clearly $\mu_{x,f}$ depends only on $x$ and $df_x$.  When $(x,\xi) \in T^*M$ we let $\mu_{x,\xi}$ denote the functor associated to the point $x$ and the affine function whose derivative at $x$ is $\xi$.

\begin{definition}
For each $F \in \Sh(M)$ we define $\SS(F) \subset T^*M$, the \emph{singular support} of $F$, to be the closure of the set of all $(x,\xi) \in T^*M$ such that $\mu_{x,\xi} F \neq 0$.
\end{definition}

As $\mu_{x,\xi} = \mu_{x,t\cdot \xi}$ when $t > 0$, the set $\SS(F)$ is \emph{conical}; that is, if $(x,\xi) \in \SS(F)$ and $t \in \bR_{>0}$, then $(x,t \cdot \xi) \in \SS(F)$.  If $F$ is locally constant away from points $\{x_i\}_{i \in I}$, then $\SS(F)$ is 
contained in the union of the zero section and the vertical cotangent spaces $T_{x_i}^* M$.  In particular $\SS(F)$ is 1-dimensional and therefore a Lagrangian subset of $T^*M$ with its usual symplectic form---we say that $\SS(F)$ is a \emph{conical Lagrangian} in $T^* M$.

\begin{definition}
Suppose $\Lambda \subset T^*M$ is a conical Lagrangian.  Define $\Sh(M,\Lambda) \subset \Sh(M)$ to be the full triangulated subcategory of sheaves with $\SS(F) \subset \Lambda$.
\end{definition}

\begin{example}
Let $\Lambda = M \cup T_{s_1}^*M \cup \cdots \cup T_{s_n}^* M$ be the union of the zero section and the cotangent spaces of finitely many points $\{s_1,\ldots,s_n\}$.  Then $\Sh(M,\Lambda)$ is the category of sheaves that are locally constant away from $\{s_1,\ldots,s_n\}$.

%If $S = \{s_1,\ldots,s_n\} \subset M$ is a finite set and $\Lambda = M \times \{0\} \cup T_{s_1}^* M \cup \ldots T_{s_n}^* M$ is the union of the zero section and the cotangent spaces to points of $S$, then 

%Let $\Lambda_\cS = \bigcup_{\alpha\in A}T^*_{S_\alpha}X$ be the union of conormals bundles
%of $S_\alpha.$  Then $\Sh(X,\cS) = \Sh(X;\Lambda_\cS).$  A microsupport condition $\Lambda$ is therefore a refinement of the notion of a stratification.
%For example, when $X=\bR$ and $\cS = \{\cdot, \bR\setminus \cdot\}$, $\Lambda_\cS = \{xy=0\}\subset
%T^*\bR,$ which looks like the symbol $+,$ and $\Sh(\bR,\cS)=\Sh(\bR;+)$ is representations of the quiver
%$\bullet \leftarrow \bullet \rightarrow \bullet$ as in Example \ref{shr+}.
\end{example}

\subsubsection{Microlocalization}
\label{subsubsec:micro}

One expects that a sheaf $F$ on $M$ has a local nature on $T^* M$ as well as on $M$, and that the singular support of $F.$  Kashiwara and Schapira formulate this by constructing for each open subset $U \subset T^*M$ a category $D^b(M;U)$ and a restriction functor from constructible sheaves on $M$ to $D^b(M;U)$.  If $F$ has singular support $\Lambda$ then the image of $F$ in $D^b(M;T^*M - \Lambda)$ is zero.

\begin{definition}[{\cite[Section 6.1]{KS}}]
Let $M$ be a one-dimensional manifold and let $U \subset T^*M$ be an open subset.  The category $D^b(M;U)$ of \emph{microlocal sheaves on $U$} is the Verdier quotient of $D^b(M)$ by the thick subcategory of all sheaves $F$ with $\SS(F) \cap U = \varnothing$.
\end{definition}

\begin{remark}
We will only make use of this construction when $U$ is conical, i.e. invariant under the action of $\bR_{>0}$.
\end{remark}

There is a dg enhancement of this construction, using for instance Drinfeld's localization construction \cite{Dr}, giving us a presheaf of dg categories on $T^* M$.  We define $\MSh$ to be the sheafification of this presheaf, and write $\MSh(M,U)$ for the value of $\MSh$ on an open subset $U \subset M$.

\begin{remark}
For each open subset $V \subset M$, we have a tautological equivalence $D^b(V) \to D^b(V;\pi^{-1}(V))$---the dg version of these assemble to a map $\Sh \to \pi_*\MSh$ in $\Shv(M;\dgCat)$. 
By definition this map induces an equivalence on stalks, and is therefore an equivalence of sheaves.  It follows that, whenever $U$ contains its projection $\pi(U)$ to the zero section of $M$,  $\MSh(M,U)$ is just the usual category of sheaves on $\pi(U)$.
\end{remark}

\begin{remark}
Fix a conical Lagrangian $\Lambda \subset T^*M$.  We can compute the full subcategory $\MSh(M,U;\Lambda) \subset \MSh(M,U)$ whose singular support is in $\Lambda \cap U$ very explicitly:
\begin{enumerate}
\item If $U$ is of the form $\pi^{-1}(I)$ where $I \subset M$ is an open set, then the restriction functor
$\Sh(M;\Lambda) \to \Sh(I;\Lambda \cap \pi^{-1}(I))$ induces an equivalence $\MSh(U;\Lambda) \cong \Sh(I;\Lambda \cap \pi^{-1}(I))$.
\item If $U$ does not contain the zero section and its intersection with $\Lambda$ consists of the vertical segments $(x_1,\bR_{> 0} \cdot \xi_1), \ldots, (x_n,\bR_{>0} \cdot \xi_n)$ then the microlocal stalk functors assemble to a map $\Sh(M;\Lambda) \stackrel{\bigoplus \mu_{x_i,\xi_i}}{\longrightarrow} \bigoplus_{i = 1}^n \bC\text{-mod}$ that induces an equivalence $\MSh(U;\Lambda) \cong \bigoplus_{i=1}^n \bC\text{-mod}$.
\end{enumerate}

\end{remark}

%\begin{definition}[{\cite[Chapter 6]{KS}}]
%Let $M$ be a one-dimensional manifold and let $\Lambda \subset T^*M$ be a conical Lagrangian subset.  If $U$ is an open subset of $T^* M$, define
%$$\MSh(M,\Lambda;U) = \Sh(M,\Lambda)/\cC$$
%where $\cC \subset \Sh(M;\Lambda)$ is the full subcategory of sheaves with $\SS(F) \cap U = \varnothing$, and the quotient construction is as in \cite{Dr}.

%let $\MSh(M,\Lambda,U)$ be the dg quotie
%\end{definition}

%\begin{remark}
%It can be shown that the assignment $U \mapsto \MSh(M,\Lambda,U)$ extends to a sheaf of triangulated dg categories on $T^*M$, supported on $\Lambda$.
%\end{remark}

%\begin{remark}
%We can compute $\MSh(M,\Lambda;U)$ very explicitly:
%\begin{enumerate}
%\item If $U$ is of the form $\pi^{-1}(I)$ where $I \subset M$ is an open set, then the restriction functor
%$\Sh(M;\Lambda) \to \Sh(I;\Lambda \cap \pi^{-1}(I))$ induces an equivalence $\MSh(M,\Lambda;U) \cong \Sh(I;\Lambda \cap \pi^{-1}(I))$.
%\item If $U$ does not contain the zero section and its intersection with $\Lambda$ consists of the vertical segments $(x_1,\bR_{> 0} \cdot \xi_1), \ldots, (x_n,\bR_{>0} \cdot \xi_n)$ then the microlocal stalk functors assemble to a map $\Sh(M;\Lambda) \stackrel{\bigoplus \mu_{x_i,\xi_i}}{\longrightarrow} \bigoplus_{i = 1}^n \Rmod$ that induces an equivalence $\MSh(M,\Lambda;U) \cong \bigoplus_{i=1}^n \Rmod$.
%\end{enumerate}
%\end{remark}

\subsubsection{Contact transformations}
\label{sec:cont}

\begin{definition}
Let $M_1$ and $M_2$ be one-dimensional manifolds, and let $U_1\subset T^*M_1$ and $U_2 \subset T^*M_2$ be conical open subsets.  A \emph{contact transformation} from $U_1$ to $U_2$ is an open immersion $f:U_1 \to U_2$ satisfying the following properties:
\begin{enumerate}
\item $f^* \omega_2 \cong \omega_1$, where  $\omega_i$ is the natural symplectic form on $T^* M_i$
\item $f$ is equivariant for the $\bR_{>0}$-actions on $U_i$, i.e. $f(t\cdot u) = t \cdot f(u)$ for $t \in \bR_{>0}$. 
\end{enumerate}
\end{definition}

If we endow $T^* \bR$ with the coordinates $(x,\xi)$ with $\xi(x) = 1$, and $U_1$ and $U_2$ are open subsets of $T^*\bR$, then any contact transformation is of the form $(x,\xi) \mapsto (f(x),\xi/f'(x))$ for some smooth function $f$.  We say that a contact transformation is \emph{affine} if $f$ is affine. More generally, if $f:M \to N$ is an affine map between affine 1-manifolds, we let $C_f$ denote the associated contact transformation, which in local coordinates as above looks like $(f(x),\xi/f'(x))$.  If $U \subset T^*M$ is connected, or more generally if the projection map $U \to M$ is injective on connected components, then every contact transformation is of the form $C_f$.

\begin{remark}
The simple structure of our contact transformations is a special feature of one-dimensional manifolds and their cotangent bundles.  A general contact transformation need not be induced by a map between the base manifolds.
\end{remark}

%\begin{definition}
%A \emph{pregraph} is a triple $(M,\Lambda,U)$ where $M$ is a one-dimensional manifold, $\Lambda \subset T^*M$ is a conical Lagrangian, and $U \subset T^*M$ is a conical open subset.  A \emph{contact transformation} between pregraphs $(M_1,\Lambda_1,U_1) \to (M_2,\Lambda_2,U_2)$ is an open immersion $f:U_1 \to U_2$ satisfying the following properties:
%\begin{enumerate}
%\item $f^* \omega_2 = \omega_1$, where $\omega_i$ is the natural symplectic form on $T^*M_i$.
%\item $f(t \cdot u) = t \cdot f(u)$ for $t \in \bR_{>0}$
%\item $f(\Lambda_1) \subset \Lambda_2$
%\end{enumerate}
%\end{definition}

We define a category $\Cont$ in the following way:
\begin{itemize}
\item The objects of $\Cont$ are triples $(M,U,\Lambda)$ where $M$ is an affine 1-manifold,  $U \subset T^*M$ is a conical open set that, and $\Lambda \subset T^*M$ is a conical Lagrangian set.  
We furthermore assume that, if $\pi:T^*M \to M$ denotes the projection, that $\pi(U) = M$.  
\item $\Hom((M_1,U_1,\Lambda_1),(M_2,U_2,\Lambda_2))$ is the set of affine contact transformations $U_1 \to U_2$ that carry $\Lambda_1$ homeomorphically onto an open subset of $\Lambda_2 \cap U_2$.
\end{itemize} 
We often abuse notation and write $(U,\Lambda)$ instead $(M,U,\Lambda)$ for an object of $\Cont$, suppressing the zero section $M$ of the ambient cotangent manifold.

\begin{remark}
\label{rem:forgetcontact}
It is easily checked that every object $(M,U,\Lambda)$ of $\Cont$ is isomorphic to an object $(M',U',\Lambda)$ for which the projection $U' \to M$ induces a bijection on connected components.  We therefore have up to unique isomorphism a ``forgetful'' functor $\Cont \to \cM$ that sends a tuple $(M,U,\Lambda)$ with $\pi_0(U) \stackrel{\sim}{\to} \pi_0(M)$ to the tuple  $(M,U,\Lambda) \mapsto M$ that sends a contact transformation of the form $C_f$ to $f$.
\end{remark}

We endow $\Cont$ with a Grothendieck topology by letting $\{(M_i,U_i,\Lambda_i)\}$ be a covering sieve of $(M,U,\Lambda)$ if the $U_i$ cover $U$.  The theory of \cite[Chapter 7]{KS} may be used to extend the sheaf $\MSh$ on each cotangent bundle $T^*M$ to a sheaf on the category $\Cont$.  In our 1-dimensional, affine setting this can be somewhat simplified---we have no need to appeal to sheaf operations defined by microlocal kernels.

\begin{enumerate}
\item The functor $\cM \to \dgCat:M \mapsto \Sh(M)$ pulls back to a functor $\Cont \to \dgCat$ along the ``forgetful'' map $\Cont \to \cM$ of Remark \ref{rem:forgetcontact}.  Let us call this functor $\cP$.
\item If we let $\cP'(M,U,\Lambda) \subset \cP(M,U,\Lambda) = \Sh(M)$ denote the full subcategory of sheaves $F$ with $\SS(F) \cap U = \varnothing$ then the assigment $(M,U,\Lambda) \mapsto \cP'$ defines a full subfunctor $\cP' \subset \cP$.
\item The quotient construction gives us a presheaf $\cP/\cP'$ on $\Cont$.  This functor has a further subfunctor $\cP'' \subset \cP/\cP'$ with $\cP''(M,U,\Lambda)$ given by the full subcategory spanned by sheaves $F$ with $\SS(F) \subset \Lambda$.
\item We let $\MSh$ denote the sheafification of the presheaf $\cP''$.
\end{enumerate}

If we restrict $\MSh:\Cont \to \dgCat$ to $\cM$ along the functor $\cM \to \Cont:(M,\Lambda) \mapsto (M,T^*M,\Lambda)$, we get the functor called $\MSh$ in section \ref{subsubsec:micro}.

\subsection{Chordal ribbon graphs}

In this section we will consider a special class of ribbon graphs.  

\begin{definition}
\label{def:chordal}
A \emph{chordal ribbon graph} is a pair $(X,Z)$ where
\begin{itemize}
\item $X$ is a ribbon graph.
\item $Z$ is a closed, bivalent subgraph containing each vertex of $X$.
\item Let $\cR_v$ denote the ternary relation defining the cyclic order on the set of half-edges incident with $v$.  If $e$ and $f$ are the two half-edges of $Z$ incident with $v$, then there is at most one half-edge $g$ so that $(e,g,f) \in \cR_v$ and at most one half-edge $h$ so that $(f,h,e) \in \cR_v$.
\end{itemize}
In particular, the last condition requires that each vertex of a chordal ribbon graph has degree at most 4.  We refer to $Z$ as the \emph{zero section} of the chordal ribbon graph.
\end{definition}

\begin{remark}
If $X$ is compact, the subset $Z$ is a disjoint collection of circles.  These circles are joined by edges of $X$ that we might call ``chords.''  Chordal ribbon graphs are similar to the ``chord diagrams'' and ``string diagrams'' of Chas and Sullivan \cite{Sullivan}.
\end{remark}

Let $\Fishn$ denote the category whose objects are chordal ribbon graphs, and where $\Hom((U,W),(X,Z))$ is given by the set of open immersions $j:U \hookrightarrow X$ with $j(W) \subset Z$ and preserving the cyclic orders at each vertex.  We endow $\Fishn$ with a Grothendieck topology in the evident way.

\begin{example}
\label{ex:fishnet1}
Let $e$ be a ribbon graph with no vertices and one edge.  We may endow it with two non-isomorphic chordal structures: one in which the zero section is empty and one in which the zero section is all of $e$.
\end{example}

\begin{example}
\label{ex:fishbone}
Let $M$ be an affine 1-manifold and $\Lambda \subset T^*M$ a conical Lagrangian subset.  We may regard $\Lambda$ as a graph by letting the affine structure on $M$ induce coordinate functions on those edges of $\Lambda$ contained in the zero section, and giving the rest of edges arbitrary some fixed coordinate $(x_0,\xi) \mapsto 1/(1+ |\xi|^2)$.  

$\Lambda$ has a ribbon structure coming from the orientation of $T^*M$.  If $\Lambda$ contains the zero section $M \subset T^*M$ then the pair $(\Lambda,M)$ is a chordal ribbon graph.  Slightly more generally, if $U \subset T^*M$ is a conical open subset then $(\Lambda \cap U,U \cap M)$ is also a chordal ribbon graph with possibly empty zero section.  We refer to chordal ribbon graphs that arise in this way as \emph{fishbones}.
\end{example}

In $\Fishn$, the full subcategory of fishbones is a basis for the Grothendieck topology.  Suppose $(\Lambda_1 \cap U_1,U_1 \cap M_1)$ and $(\Lambda_2 \cap U_2,U_2 \cap M_2)$ are two fishbones arising from $(M_1,U_1,\Lambda_1)$ and $(M_2,U_2,\Lambda_2)$.  Each morphism of chordal ribbon graphs induces a morphism $(M_1,U_1,\Lambda_1) \to (M_2,U_2,\Lambda_2)$ in $\Cont$, and this functor respects the Grothendieck topologies.  By Proposition \ref{prop:base}, this shows that $\MSh \in \Shv(\Cont,\dgCat)$ determines a sheaf of categories $\CPM$ on $\Fishn$. 

\begin{definition}
We let $\CPM:\Fishn \to \dgCat$ denote the sheaf of dg categories on $\Fishn$ whose restriction to $\Cont$ is given by $\MSh$ of Section \ref{sec:cont}.  We call $\CPM(X,Z)$ the \emph{constructible plumbing model} of the chordal
ribbon graph $(X,Z)$.
\end{definition}

\subsection{Quiver descriptions of microlocal categories}

Let $M$ be a one-dimensional manifold and let $\Lambda \subset T^* M$ be a conic Lagrangian containing the zero section.  The category $\Sh(M;\Lambda)$ can be described very concretely in terms of representations of quivers, which we will recall in this section.  We also discuss the Bernstein-Gelfand-Ponomoraev equivalences that have some relevance for us.

Let us call connected components of $\Lambda - M$ the \emph{spokes} of $\Lambda$.  They are divided into two groups depending on which component of $T^*M - M$ they fall into.  Using an orientation of $M$ we may label these groups ``upward'' and ``downward.''

The conic Lagrangian $\Lambda$ determines a partition $P_\Lambda$ of $M$ into subintervals (which may be open, half-open, or closed) and points.  Let us describe this partition in case $M = \bR$, the general case is similar.  Each spoke of $\Lambda$ is incident with a point $x \in \bR$, which we may order $x_1 < \ldots < x_k$.  We put $\{x_i\} \in P_\Lambda$ if $x_i$ is incident with both an upward and a downward spoke.  We put an interval $I$ from $x_i$ to $x_{i+1}$ in $\cP_\Lambda$ whose boundary conditions are determined by the following rules
\begin{itemize}
\item If $x_i$ is incident with an upward spoke but not incident with a downward spoke, then $x_i$ is included in $I$.  Otherwise $x_i$ is not included in $I$.
\item If $x_{i+1}$ is incident with a downward spoke but not incident with an upward spoke, then $x_{i+1}$ is included in $I$.  Otherwise $x_{i+1}$ is not included in $I$.
\end{itemize}
We put $(-\infty,x_1)$ in $P_\Lambda$ if $x_1$ is incident with an upward spoke and $(-\infty,x_1]$ in $P_\Lambda$ if $x_1$ is incident with a downward spoke, and similarly we put $(x_k,\infty)$ (resp. $[x_k,\infty)$) in $P_\Lambda$ if $x_k$ is incident with a downward (resp. upward) spoke.

Define a quiver (that is, directed graph) $Q_\Lambda$ whose vertices are the elements of $P_\Lambda$ and with and edge joining $I$ to $J$ (in that orientation) if the closure of $J$ has nonempty intersection with $I$.  If there are $n$ spokes then this is a quiver of type $A_{n+1}$ (i.e. shaped like the Dynkin diagram $A_{n+1}$) whose edges are in natural bijection with the spokes of $\Lambda$: an upward spoke corresponds to a left-pointing arrow and a downward spoke to a right-pointing arrow.

\begin{theorem}
\label{thm:quiverquiver}
There is a natural equivalence of dg categories
$$\Sh(M;\Lambda) \cong \Rep(Q_\Lambda)$$
If $(x,\xi)$ belongs to a spoke of $\Lambda$ corresponding to an arrow $f$ of $Q_\Lambda$, then under this equivalence the functor $\mu_{x,\xi}$ intertwines with the functor ${\mathsf Cone}(f)$.
\end{theorem}

\begin{proof}
The equivalence sends the constant sheaf on $I \subset M$ 
(concentrated in degree zero and with

\end{proof}

\begin{example}
Let $\raise3pt\hbox{$\bot$}\!\!\raise0.36pt\hbox{+}\!\!\hbox{\lower4.5pt\hbox{$\top$}}\subset T^*\bR$
be the union of the zero section, the fiber at $0,$ an upward spoke at some $x_- <0$ and a downward spoke at some $x_+>0.$  Then
$$
\Sh(\bR,\raise3pt\hbox{$\bot$}\!\!\raise0.36pt\hbox{+}\!\!\hbox{\lower4.5pt\hbox{$\top$}})\cong Rep( \bullet \leftarrow \bullet \leftarrow \bullet \rightarrow \bullet \rightarrow \bullet).
$$
\end{example}

For a general quiver $Q$, if $a$ is an arrow let $s(a)$ and $t(a)$ denote the source and target of $a$, respectively.  A vertex $v$ of $Q$ is called a sink (resp. source) if all the arrows incident to it have $t(a) = v$ (resp. $s(a) = v$).  If $x$ is a sink or a source, then Bernstein-Gelfand-Ponomaraev \cite{BGP} define a new quiver $S_x Q$ obtained by reversing the orientation of all the arrows in $Q$ incident to $x$.

\begin{theorem}[Bernstein-Gelfand-Ponomaraev \cite{BGP}]
Let $Q$ be a quiver, and let $x \in Q$ be a sink or a source.  Then there is an equivalence of dg categories
$$\Rep(Q) \cong \Rep(S_x Q)$$
If $Q_1$ and $Q_2$ are quivers with same underlying undirected graph, then $\Rep(Q_1) \cong \Rep(Q_2)$.  
\end{theorem}

Applying this theorem to the quivers obtained from fishbones we obtain the following.

\begin{corollary}
\label{cor:corcorcor}
Let $M$ be a one-dimensional manifold and let $\Lambda_1$ and $\Lambda_2$ be conic Lagrangians in $T^*M$.  Suppose that in each connected component $U$ of $T^*M - M$, $\Lambda_1 \cap U$ and $\Lambda_2 \cap U$ have an equal number of components (i.e. $\Lambda_1$ and $\Lambda_2$ have an equal number of spokes in each group.)  Then $\Sh(M;\Lambda_1) \cong \Sh(M;\Lambda_2)$.
\end{corollary}

\section{Dualizable Ribbon Graphs, Beilinson-Bondal, and HMS}

Given a chordal ribbon graph $(X,Z)$, we may form a graph $B$ in the following way: the vertices are the connected components of $Z$, and we draw an edge from a component $Z_1$ to a component $Z_2$ if there exists an edge of $X$ joining a vertex of $Z_1$ to a vertex of $Z_2$.  If there are noncompact edges incident with a vertex of $Z_1$ but no other vertex, then we add a noncompact edge to $B$ connected to the component $Z_1$.  By construction, there is a map from $X$ to $B$ which collapses each component $Z_i$ to a vertex.

\begin{definition}
\label{def:B}
A chordal ribbon graph $(X,Z)$ is called \emph{dualizable} if each vertex of the associated graph $B$ has degree 2.
\end{definition}

As discussed in the Introduction, dualizability is a combinatorial analogue of a torus fibration with a section.
In this section we show that $\CPM$ of a dualizable ribbon graph is equivalent to perfect complexes on a one-dimensional variety or orbifold---a chain of weighted projective lines (``balloons'') indexed by the vertices of $B$.

\subsection{tcnc stacks}
\label{sec:tcnc}

We shall introduce two classes of proper one-dimensional stacks with toric components, called respectively ``balloon chains'' and ``balloon rings,'' which arise as mirror partners of dualizable ribbon graphs.  We refer to Deligne-Mumford stacks of any of these two kinds as ``one-dimensional stacks with toric components and normal crossings'' or ``one-dimensional tcnc stacks'' for short.

\begin{definition}
\label{def:stacks}
For $a \in \bZ_{\geq 1}$, Let $\mu_a \subset \bC^*$ be the multiplicative group of the $a$th roots of unity.
\begin{itemize}
\item We let $B(a) = [\Spec(\bC)/\mu_a]$ denote the classfying stack of $\mu_a$.
\item Let $\bC[T]$ be a polynomial ring with coordinate variable $T$.  We let $U(a) = [\Spec(\bC[T])/\mu_a]$ denote the quotient stack associated to the action of $\mu_a$ on $\bC[T]$ given 
$$\zeta \cdot T = \zeta^{-1} T$$
\item Let $\bC[T,V]/TV$ denote the coordinate ring of the $T$- and $V$-axes in the $TV$-plane.  Let $X(a) = [\Spec(\bC[T,V]/TV)/\mu_a]$ denote the quotient stack associated to the action of $\mu_a$ on $\bC[T,V]/TV$ given by
$$\zeta \cdot T = \zeta T \qquad \zeta \cdot V = \zeta^{-1} V$$
\end{itemize}
\end{definition}

\begin{remark}
We have open inclusions
$$
\begin{array}{c}
U(a) \supset [\Spec(\bC[T,T^{-1}])/\mu_a] \\
X(a) \supset [(\Spec(\bC[T,T^{-1},V,V^{-1}])/TV)/\mu_a]
\end{array}
$$
As $\mu_a$ acts freely on $\bC - \{0\}$, we may identify these open subsets with affine schemes $\Spec(\bC[T^a,T^{-a}])$ and $\Spec(\bC[T^a,T^{-a}]) \coprod \Spec(\bC[V^a,V^{-a}])$.
\end{remark}

\begin{remark}
\label{rem:maps}
The stacks $B(1)$, $U(1)$, and $X(1)$ are representable, i.e. they are ordinary varieties.  There are \'etale maps $B(1) \to B(a)$, $U(1) \to U(a)$, and $X(1) \to X(a)$, as well as projections to coarse moduli spaces $B(a) \to B(1)$, $U(a) \to U(1)$, and $X(a) \to X(1)$.  The compositions $U(1) \to U(a) \to U(1)$ etc. are the GIT quotient maps for the actions of $\mu_a$.
\end{remark}

\begin{definition}
\label{def:balloon}
A \emph{balloon} with indices $(a_1,a_2)$ is the pushout of the diagram
$$
\xymatrix{
& \bC - \{0\} \ar[rd]^{z \mapsto T^{a_1}} \ar[ld]_{z \mapsto T^{a_2}} & \\
U(a_1) & & U(a_2)
}
$$
in the 2-category of Deligne-Mumford stacks.  That is, a balloon is a weighted projective line that is generically representable.  We will define ``balloon chains'' and ``balloon rings'' by gluing together balloons at their orbifold points.
\begin{enumerate}
\item
For $n \geq 2$, a \emph{balloon chain} with indices $(a_1,\ldots,a_n)$ is the pushout of the diagram
$$
\xymatrix{
& \bC - \{0\} \ar[ld] \ar[rd] & & \cdots \ar[ld] \ar[rd] & & \bC - \{0\} \ar[dl] \ar[dr] \\
U(a_1) & & X(a_2) & & X(a_{n-1}) &  & U(a_n)
}
$$ 
\item
For $n \geq 2$, a \emph{balloon ring} with indices $(a_1,\ldots,a_n)$ is the pushout of the diagram
$$
\xymatrix{
& \bC - \{0\} \ar[ld] \ar[rd] & & \cdots \ar[ld] \ar[rd] & & \bC - \{0\} \ar[dl] \ar[dlllll]  \\
X(a_1) & & X(a_2) & & X(a_n) %& & X(a_n)
}
$$ 

\end{enumerate}
If $A = (a_1,\ldots,a_n)$ is an $n$-tuple of positive integers, we will denote by $C(A)$ the balloon chain with indices $A$ and by $R(A)$ the balloon ring with indices $A$.
\end{definition}

\begin{remark}
If $A = (a_1,\ldots,a_n)$ is an $n$-tuple of positive integers, let $\tilde{C}(A)$ (resp. $\tilde{R}(A)$) be the disjoint union of balloons $C(a_1,a_2) \amalg \cdots \amalg C(a_{n-1},a_n)$ (resp. $C(a_1,a_2) \amalg \cdots \amalg C(a_{n-1},a_n) \amalg C(a_n,a_1)$.  There are maps $\tilde{C}(A) \to C(A)$ and $\tilde{R}(A) \to R(A)$ that exhibit $\tilde{C}$ and $\tilde{R}$ as the normalization of the Deligne-Mumford stacks $C(A)$ and $R(A)$.

%Let $\overline{1} = (1,\ldots,1)$ denote the $n$-tuple filled with ``1''s.  There are maps $p_{C(A)}$ and $p_{R(A)}$ fitting into commutative diagrams
%$$
%\xymatrix{
%\  \bP^{1} \amalg \ldots \amalg \bP^{1} = \tilde{C}(\overline{1}) \ar[r] \ar[d]^{\pi} & \tilde{C}(A)\ar[d]^{\pi} & \bP^{1} \amalg \ldots \amalg \bP^{1} = \tilde{R}(\overline{1}) \ar[r] \ar[d]^{\pi} & \tilde{R}(A) \ar[d]^{\pi} \\
%\   C(\overline{1}) \ar[r]^{p_{C(A)}}         &      C(A) , & R(\overline{1}) \ar[r]^{p_{R(A)}} & R(A).                
%}
%$$
%where the top horizontal arrows are the maps of remark \ref{rem:maps}.
\end{remark}

\subsection{Perfect complexes on balloon chains and rings}

For $A = (a_1,\ldots,a_n)$ an $n$-tuple of positive integers, set $B_C(A) = B(a_2) \amalg \cdots \amalg B(a_{n-1})$ and $B_R(A) = B(a_1) \amalg \cdots \amalg B(a_n)$.  The inclusion map $B_C(A) \hookrightarrow C(A)$ (resp. $B_R(A) \hookrightarrow R(A)$ lifts to an inclusion to $\tilde{C}(A)$ (resp. $\tilde{R}(A)$) in two different ways, which we denote by $i_1$ and $i_2$.  We have diagrams
$$
\begin{array}{c}
B_C(A) \rightrightarrows \tilde{C}(A) \to C(A)\\
B_R(A) \rightrightarrows \tilde{R}(A) \to R(A)
\end{array}
$$
(Note that these diagrams should be understood as functors from the nerve of the category $\bullet \rightrightarrows \bullet \to \bullet$ into the $\infty$-category $\St_{/\bC}$.)  In fact we will show in Proposition \ref{prop:equalizers} below that these functors are coequalizer diagrams in $\St_{/\bC}$.

\begin{proposition}
\label{prop:equalizers}
Let $A = (a_1,\ldots,a_n)$ be an $n$-tuple of positive integers.  The following diagrams are equalizers in the $\infty$-category of dg categories.
\begin{enumerate}
\item $\Perf(C(A)) \to \Perf(\tilde{C}(A)) \rightrightarrows \Perf(B(A))$
\item $\Perf(R(A)) \to \Perf(\tilde{R}(A)) \rightrightarrows \Perf(B(A))$
\end{enumerate}
\end{proposition}

This is essentially proved in \cite[Section 4]{La}, but let us give a proof in our current language:

\begin{proof}
Suppose we have a diagram of stacks $X'' \rightrightarrows X' \to X$, and an \'etale cover $\{U_i \to X\}$ of $X$.  Set $U'_i = U_i \times_X X'$ and $U''_i = U_i \times_X X''$, so that for each $i$ we have a diagram
$$U''_i \rightrightarrows U'_i \to U_i$$
By Proposition \ref{prop:descent} and basic properties of limits, to show that $\Perf(X) \to \Perf(X') \rightrightarrows \Perf(X'')$ is an equalizer diagram it suffices to show that$\Perf(U_i) \to \Perf(U'_i) \to \Perf(U''_i)$ is an equalizer diagram for each $i$.  Applying this remark to the open covers $\{X(a_1),\ldots,X(a_n)\}$ of $R(A)$ and $\{U(a_1),X(a_2),\ldots,X(a_{n-1}),U(a_n)\}$ of $C(A)$, we see that the Proposition reduces to the claim that the map $\Perf(X(a)) \to \Perf(\tilde{X}(a)) \rightrightarrows \Perf(B(a))$ is an equalizer, where $\tilde{X}(a) = U(a) \amalg U(a)$ is the normalization of $X(a)$.  Another application of Proposition \ref{prop:descent}, this time to the \'etale cover $X(1) \to X(a)$, further reduces us to the case where $a = 1$.

Set $X = X(1) = \Spec(\bC[x,y]/xy)$ and $\tilde{X} = \Spec(\bC[x] \coprod \bC[y])$, and $B = \Spec(\bC)$.  To show that the equalizer $\mathbf{E}$ of $\Perf(\tilde{X}) \to \Perf(X)$ is equivalent to $\Perf(X)$ it suffices to show that the image $I$ of $\cO_X$ in $\mathbf{E}$ is a generator, and that $\mathbf{R}\mathrm{Hom}(\cO_X,\cO_X) \cong \mathbf{R}\mathrm{Hom}(I,I)$.  An object of $\mathbf{E}$ is determined by
\begin{itemize}
\item A complex of vector bundles $\cE$ on $\tilde{X}$
\item A complex of vector spaces $V$
\item A pair of quasi-isomorphisms between the fiber of $\cE$ at $x = 0$ and the fiber of $\cE$ at $y = 0$.
$$\alpha:\cE_{x = 0} \stackrel{\sim}{\to} V \stackrel{\sim}{\leftarrow} \cE_{y = 0}:\beta$$
\end{itemize}
The object $I$ corresponds to the data $\cE = \cO_{\tilde{X}}$ and $V = \bC$.  To show that $I$ generates we may reduce by induction to the case where $\cE$ is a chain complex of length 1.  In that case $V$ is quasi-isomorphic to a complex of length $1$ as well, and we see that $\cE \cong I^{\oplus r}$ for some $r$ because every vector bundle on $\tilde{X}$ is trivializable.  The ring complex $\bfR\Hom(I,I)$ is the equalizer (in the homotopy sense) of the diagram
$$\bfR\Hom(\cO_{\tilde{X}},\cO_{\tilde{X}}) \rightrightarrows \bfR\Hom(\bC,\bC)$$
which is the same as
$$\bC[x] \times \bC[y] \rightrightarrows \bC$$
One computes that the homotopy equalizer is concentrated in a single degree and agrees with the subring of $\bC[x] \times \bC[y]$ identified with $\bC[x,y]/xy$.
\end{proof}

\subsection{Wheels and the Beilinson-Bondal equivalence}

A \emph{wheel} is a conical Lagrangian $\Lambda$ in $T^* S^1$ that contains the zero section.  It has a canonical chordal ribbon structure whose zero section $Z$ is the zero section of $T^*S^1$.  We call the connected components of the complement of $Z$ in $\Lambda$ the \emph{spokes} of the wheel.  They are divided into two groups depending on which component of $T^* S^1 - S^1$ they belong to.  Suppose there are $a$ spokes on one side and $b$ spokes on the other side.  Bondal \cite{B} constructed an equivalence
$$\Perf(C(a,b)) \cong \Sh(S^1;\Lambda)$$
This equivalence is very much in the spirit of an old result of Beilinson on the derived category of projective space, which in particular showed that $\Perf(\bP^1)$ was equivalent to the category of representations of the quiver $\bullet \rightrightarrows \bullet$.

\begin{theorem}[Beilinson-Bondal]
\label{thm:BB}
Let $\Lambda \subset T^* S^1$ be a wheel, and let $U_1 \subset T^* S^1$ and $U_2 \subset T^* S^1$ be the two connected components of $T^* S^1 - S^1$.  Suppose that there are $a_1$ spokes in $U_1$ and $a_2$ spoked in $U_2$.  There is a commutative diagram of triangulated dg categories
$$
\xymatrix{
\Perf(B(a_1)) \ar[d]^{\cong} & \ar[l] \Perf(C(a_1,a_2)) \ar[r] \ar[d]^{\cong} &\Perf(B(a_2)) \ar[d]^{\cong}\\
\MSh(U_1,\Lambda) & \ar[l] \Sh(S^1;\Lambda) \ar[r] & \MSh(U_2,\Lambda)}
$$
\end{theorem}

\begin{proof}
By Corollary \ref{cor:corcorcor} we may assume $\Lambda$ has any convenient shape so long as we do not change the numbers $a_1$ and $a_2$.  Let us identify the base manifold $S^1$ with the unit circle in the $\bC$, and let $\Lambda$ be the union of $S^1$ together with the $a_1$ upward spokes placed at $a_1$th roots of unity and $a_2$ downward spokes placed at $a_2$th roots of unity.  (This is a wrapped-up version of a special case of the conic Lagrangians considered in \cite{fltz:orb}.)  Let $Q_\Lambda$ be the corresponding quiver.  By Theorem \ref{thm:quiverquiver} we only have to show that $\Perf(C(a_1,a_2))$ is equivalent to $\Rep(Q_\Lambda)$.  This follows from \cite{fltz:orb}.

\end{proof}

We may use this to deduce the main result of this paper.  Let $X$ be a dualizable ribbon graph and let $B$ be the associated ``base'' graph of Definition \ref{def:B}.  If $X$ is connected then $B$ is either a cycle or a path---let us number the edges of $B$ by the integers $1$-$n$ in a natural way, such a way that the $i$th edge and the $(i+1)$st edge (and the $n$th and $1$st edge, if $B$ is a cycle) share a common vertex.  Then let $a_i$ be the number of edges of $X$ that lie above the $i$th edge.  Call $(a_1,\ldots,a_n)$ the indices associated to the dualizable ribbon graph.

\begin{theorem}[HMS]
Let $X$ be a dualizable ribbon graph with indices $(a_1,\ldots,a_n)$.  
\begin{enumerate}
\item If the base graph $B$ is a path, then there is an equivalence of dg categories $\CPM(X) \cong \Perf(C(a_1,\ldots,a_n))$
\item If the base graph $B$ is a cycle, then there is an equivalence of dg categories
$\CPM(X) \cong \Perf(R(a_1,\ldots,a_n))$
\end{enumerate}
\end{theorem}

\begin{proof}
In case (1), there is an open cover of $(X,Z)$ by charts $(W_i,Z_i)$, $i = 1, \ldots, n$, such that $W_i$ is a wheel, $W_i \cap W_j = \varnothing$ if $|i-j| \geq 2$, and $W_i \cap W_{i+1}$ is a component of $W_i - Z$ (and a component of $W_{i+1} - Z$).  Then by the sheaf property of $\CPM$ we have an equalizer diagram
$$\CPM(X) \to \CPM(\coprod W_i) \rightrightarrows \CPM(\coprod W_i \cap W_{i+1})$$
The Theorem then follows immediately from Theorem \ref{thm:BB} and Proposition \ref{prop:equalizers}.  Case (2) is similar.
\end{proof}

\section{The category of ribbon graphs and partial contractions}

The motivation for our construction of $\cpm$ -- the idea that the skeleton encodes the Fukaya category --
means that the category we assign to a ribbon graph should not depend on the chordal structure.  As a result,
the microlocal categories $\Sh(M;\Lambda)$ should have extra symmetries coming from symplectomorphismsms of $T^*M$ that do not necessarily respect the ``cotangent'' structure.  For instance, if $\Lambda \subset T^*\bR$ is the union of the zero section and the ``northward'' ray at zero, the graph is trivalent and suggests a three-fold rotational symmetry (or more precisely a lift thereof, since a $2\pi$ rotation shifts the grading by $2$).

This observation is related to a deficiency of $\CPM,$ as we have defined it, when viewed as a model for the Fukaya category of a symplectic surface.  Isomorphisms between chordal ribbon graphs are more rigid than symplectomorphisms between surfaces.  For instance, if $X$ and $Y$ are the two chordal ribbon graphs of Example \ref{ex:fishnet1}, then $\CPM(X)$ and $\CPM(Y)$ are equivalent but this equivalence is not witnessed by any map between $X$ and $Y$ in the category $\Fishn$.
We can resolve this issue by adding morphisms to the category of chordal ribbon graphs.  First, we should allow our morphisms to ignore the ``zero section'' of the chordal ribbon
graph, so that the objects in our category might as well be plain ribbon graphs.  Next, we should allow morphisms between ribbon graphs that contract edges (thus increasing the degrees of the vertices in our graphs beyond four).

In this section we discuss and define what such an enhanced $\CPM$ should look like: it should be a sheaf on a category of ribbon graphs (actually, for standard symplectic-geometric reasons, structures we refer to as ``graded ribbon graphs'') $\RGpc$ that admits a functor from $\Fishn$ and whose restriction to $\Fishn$ is the usual $\CPM$.

\subsection{The category $\RGpc$}

We will discuss two classes of morphisms between ribbon graphs:

\begin{definition}
Let $X$ and $Y$ be ribbon graphs.
\begin{enumerate}
\item An \emph{open immersion} $j:X \hookrightarrow Y$
is an open immersion of the underlying graphs with the property that the map induced by $j$ from the set of half-edges of $X$ incident with $v$ to the set of half-edges of $Y$ incident with $j(v)$ preserves the cyclic order.
\end{enumerate}
For each vertex $v \in Y$ we define the \emph{star} of $v$ to be the open subgraph consisting of $v$ and all the edges incident with $v$.  Write $\mathrm{star}(v) \subset Y$ for the star of $v$.
\begin{enumerate}
\item[(2)]
Let $X$ and $Y$ be ribbon graphs.  A \emph{simple contraction} $p:X \to Y$ is a is a morphism $p:X \to Y$ of underlying ribbon graphs with the property that, for all vertices $v$ of $Y$,  $p^{-1}(\mathrm{star}(v))$ is a ribbon tree and the map from the leaves of $p^{-1}(\mathrm{star}(v))$ to the leaves of $\mathrm{star}(v)$ preserves the cyclic order of Proposition \ref{prop:jointree}.
\end{enumerate}
\end{definition}

\begin{definition}
Let $X$ and $Y$ be ribbon graphs.  A \emph{partial contraction} from $X$ to $Y$ is a diagram
$$X \stackrel{j}{\hookleftarrow} U \stackrel{f}{\rightarrow} Y$$
where $j$ is an open inclusion of ribbon graphs and $f$ is a simple contraction.  We can compose a partial contraction $X \supset U \to Y$ with $Y \supset V \to Z$ by setting $W \subset U$ to be the inverse image of $V$ under $U \to Y$, and mapping $W$ to $X$ and $Y$ in the evident ways.  In this way ribbon graphs form a category $\RGpc$ whose morphisms are partial contractions.
We will denote by $\RTpc$ the full subcategory of $\RGpc$ spanned by ribbon trees.
\end{definition}

We will endow $(\RGpc)^{\op}$ with a Grothendieck topology, by singling out a family of covering sieves for each ribbon graph $X$.  
\begin{definition}
\label{def:coveringsieves}
Let $\cU = \{U \stackrel{p}{\leftarrow} U' \stackrel{i}{\hookrightarrow} X\}$
be a sieve on $X \in \RGpc^{\op}$.  We say that $\cU$ is a \emph{covering sieve} if each vertex of $X$ is in the image of one of the $U' \hookrightarrow X$.
\end{definition}

\begin{proposition}
The collection of covering sieves of definition \ref{def:coveringsieves} satisfies the axioms for a Grothendieck topology.  The full subcategory $(\RTpc)^{\op} \subset (\RGpc)^{\op}$ is a basis in the sense of Proposition \ref{prop:base}.
\end{proposition}

\begin{proof}
To prove the first statement, we must verify the three conditions of Definition \ref{def:gt}.
The first condition is immediate, as the over-category places no restrictions on morphisms,
and $id_X$ establishes the covering property.
For the second, 
let $f:Y\rightarrow X$ be a morphism in $(\RGpc)^{\op},$ which we shall write
$f: Y\stackrel{f_p}\leftarrow Y' \stackrel{f_i}\hookrightarrow X.$
Then if $g:W\rightarrow Y$ is another morphism, the composition is
$f\circ g: W \stackrel{g_p\circ g_i^{-1}\circ f_p}{\longleftarrow}f_p^{-1}g_i(W')\stackrel{f_i}\hookrightarrow X$.
It is simple to check that $f\circ(g\circ h) = (f\circ g)\circ h,$ which establishes
the pull-back condition, i.e. $h:V\rightarrow W$ in $\RGpc^{\op}$ and $g\in f^*\cU$ implies $g\circ h \in f^*\cU.$
The covering property is immediate.
For the third condition, suppose that $\cU$ is a covering sieve on $X,$ $\cV$ is any sieve on $X,$
and for all $f:Y\rightarrow X$, $f^*\cV$ is a covering sieve on $Y$.  We must show that $\cV$ is
a covering sieve.  Since $\cV$ is already a sieve, we must only establish the covering property,
namely that all vertices lie in the image of morphisms of $\cV$.  So let $v\in V_X$ and put $f:  Star(v) \rightarrow X$.
Then $f^*\cV$ is a covering sieve, meaning $v$ is in the image of
a morphism of $\cV,$ i.e. $\cV$ covers $X$.  This establishes the first statement.

To prove the second statement, we must show that every graph $X$ admits a covering sieve
by trees.  Define the sieve $\cT_X = \{T\leftarrow T' \hookrightarrow X\}$ where $T$
is a tree (and necessarily $T'$ as well).  This is a sieve, since if $T\in (\RGpc)^{\op}$ is a tree, any
partial contraction $S\leftarrow S'\hookrightarrow T$ in $(\RGpc)^{\op}$ implies that $S$ is
a tree as well, since it is (a contraction of) an open subgraph of a tree.
It is a covering sieve, since every vertex $v\in V_X$ lies in the image of
the morphism $Star(v)\rightarrow X$ in $\cT_X.$  This completes the proof.\
\end{proof}

\subsection{Fukaya formalisms, $\bZ/2$-graded case}

To define the Fukaya category in symplectic geometry, it's necessary to find a coherent way to grade all of the Floer cohomology groups.  This requires a little bit of extra structure on the symplectic manifold, called a \emph{grading}.  In the setting of ribbon graphs, these structures have combinatorial descriptions.

\begin{definition}
A \emph{$\bZ/2$-graded} ribbon graph is a ribbon graph equipped with a $\bZ/2$-torsor $\tau$ on the underlying space.
\end{definition}

A morphism of $\bZ/2$-graded ribbon graphs is a morphism $f:X \to Y$ of the underlying ribbon graphs together with an isomorphism $f^* \tau_Y \cong \tau_X$.  We define partial contractions as before.  Write $\RGpc^{\bZ/2}$ for the category of $\bZ/2$-graded ribbon graphs and partial contractions.

\begin{remark}
Seidel defined \cite{S-gradings} for $N = 1,2,\ldots, \infty$, a $\bZ/N$-extension $\Sp^N(2n)$ of the symplectic group $\Sp(2n;\bR)$.  When $N = 2$, we have $\Sp^N(2n) = \bZ/2 \times \Sp(2n)$.  Just as a graph embedded in a symplectic surface---i.e. a surface whose tangent bundle has structure group reduced to $\Sp(2;\bR)$---inherits a ribbon structure, a graph embedded in a ``2-graded symplectic surface''---i.e. a surface whose tangent bundle has structure group reduced and lifted to $\Sp^2(2;\bR)$---inherits a $\bZ/2$-graded ribbon structure.
\end{remark}

We endow $(\RGpc^{\bZ/2})^{\op}$ with the coarsest Grothendieck topology that makes the forgetful functor $(\RGpc^{\bZ/2})^{\op} \to \RGpc^{\op}$ continuous.

\begin{definition}
A $\bZ/2$-graded \emph{Fukaya formalism} is a functor from $\RGpc^{\bZ/2}$ to $\dgCat$ that satisfies the sheaf condition for the Grothendieck topology on $(\RGpc^{\bZ/2})^{\op}$.
\end{definition}

\begin{remark}
A functor $\Fuk:\RGpc^{\bZ/2} \to \dgCat$, can be described a little more informally and prosaically in the following way.  Such a functor consists of the following data:
\begin{itemize}
\item A dg category $\Fuk(X)$ for each $\bZ/2$-graded ribbon graph $X$.
\item A contravariant assignment from open inclusion $j:U \to X$ to functors $j^*:\Fuk(X) \to \Fuk(U)$
\item A covariant assignment from simple contractions $u:X \to Y$ to equivalences $u_!:\Fuk(X) \to \Fuk(Y)$.
\end{itemize}
The maps are subject to two evident compatibilities (between $j^* \circ k^*$ and $(k \circ j)^*$; between $u_! \circ v_!$ and $(u \circ v)_!$) as well as the following slightly more subtle base-change compatibility.  Whenever we have a commutative square of the form
$$\xymatrix{
f^{-1}(U) \ar[r]^{g} \ar[d]_{j} & U \ar[d]^{i} \\
X \ar[r]_{f} & Y
}$$
whose horizontal arrows are contraction morphisms and whose vertical arrows are open inclusions, then the functors $i^* \circ f_!$ and $g_! \circ j^*$ are isomorphic.  The notion of a functor from $\RGpc^{\bZ/2}$ efficiently encodes this information, as well as the higher compatibilities between the compatibilities that we have not spelled out.
\end{remark}

Let $\RGo$ denote the category of ribbon graphs and open inclusions.  We have a faithful embedding $\RGo^{\op} \hookrightarrow \RGpc$, which induces a Grothendieck topology on $\RGo$ and on each of the over-categories $(\RGo)_{/X}$.  This coincides with the usual Grothendieck topology on $(\RGo)_{/X}$ of open subgraphs of $X$.  There is a similar subcategory $(\RGo^{\bZ/2})^{\op} \hookrightarrow \RGpc^{\bZ/2}$ with a similar Grothendieck topology.

\begin{proposition}
Let $\Fuk:\RGpc \to \dgCat$ be a Fukaya formalism.  The following hold:
\begin{enumerate}
\item $\Fuk\vert_{(\RGo^{\bZ/2})^{\op}_{/X}}$ is a sheaf.  That is, $\Fuk$ determines a sheaf of dg categories on each $\bZ/2$-graded ribbon graph.
\item Whenever $X \to Y$ is a simple contraction, $\Fuk(X) \to \Fuk(Y)$ is an equivalence.
\end{enumerate}
\end{proposition}

\begin{remark}
It can be shown that conversely any functor satisfying conditions (1) and (2) is a Fukaya formalism.  The fact that the Fukaya category of a Liouville manifold (such as a punctured Riemann surface or a cotangent bundle) should have any property like (1) follows \emph{for cotangent bundles} from the work of Nadler and the last author.  Much more generally, this idea has been advanced by Kontsevich \cite{K-stein}.  Locality of the Fukaya category is also behind Abouzaid's plumbing construction \cite{A} and Seidel's construction of \cite{S-spec} --- see Section \ref{sec:influences}.
\end{remark}

\begin{proof}

Let $F$ be a functor $\RGpc^{\bZ/2} \to \dgCat$.  Condition (1) is immediate from the continuity of the functor $(\RGo)_{/X}^{\bZ/2} \to (\RGpc^{\bZ/2})^{\op}$.  To show that $F$ satisfies condition (2) we have to check that when $f:X \to Y$ is a morphism in $\RGc$ then $F(f)$ is an equivalence.  To do this, we build a sieve $\cU_f$ on $X$ out of $f$ by putting a partial contraction $U \leftarrow U' \hookrightarrow X$ in $\cU_f$ whenever there is a commutative square
$$\xymatrix{
U \ar[d] & U' \ar[d] \ar[l] \\
Y & X \ar[l]
}
$$
where the vertical maps are open inclusions.  Such a square is necessarily unique.  In particular $Y \leftarrow X \stackrel{=}{\to} X$ is the initial object of $\cU_f$.  It follows that $F(Y) \cong \varprojlim_{\cU_f} F(U)$, and the sheaf condition implies that $F(X) \to F(Y)$ is an equivalence as required.
\end{proof}

\subsection{$\bZ$-gradings}

We have just descibed a notion of ``Fukaya formalism'' for $\bZ/2$-graded ribbon graphs, suitable for modeling the Fukaya categories of exact $2$-graded symplectic surfaces.  There is a similar story for $N$-graded surfaces for every $N$.  In this section we treat the case $N = \infty$.  The following proposition can serve as a definition of an $\infty$-graded symplectic surface for those unfamiliar with \cite{S-gradings}.

\begin{proposition}
\label{prop:G}
Let $G$ be the quotient of $\bR \times \bZ$ by the relation $(2\pi,0) = (0,1)$.  Let $G \to \SL_2(\bR)$ be a homomorphism that sends $(0,1)$ to $0$ and whose image is a circle.  If $(M,\omega)$ is a symplectic surface, then the space of $\bZ$-gradings of $M$ is homotopy equivalent to the space of lifts of the symplectic structure to a $G$-torsor.
\end{proposition}

In particular, to give a $\bZ$-grading of $M$ is equivalent to giving a $G$-torsor on $M$ compatible with the symplectic structure.  In this section we give a combinatorial version of a $\bZ$-graded symplectic surface in terms of ribbon graphs.

\subsubsection{Unwindings of cyclically ordered sets}

If the essential combinatorial ingredient in the definition of a ribbon graph is the notion of a cyclic order, then the essential combinatorial ingredient in the definition of a $\bZ$-graded ribbon graph is the notion of an ``unwinding'' of cyclic order.

\begin{definition}
For each integer $n \geq 2$ let $G_n$ be the group generated by commuting symbols $R$ and $S$ subject to the relation $R^n = S^2$.  Thus, we have $G_n \cong \bZ$ if $n$ is odd and $G_n \cong \bZ \times \bZ/2$ if $n$ is even.
\end{definition}

Denote by $\rho:G_n \to \bZ/n$ the homomorphism that sends $R$ to $1 + n \bZ$ and $S$ to $0$---it exhibits $G_n$ as an extension of $\bZ/n$ by $\bZ$.  Denote by $\sigma:G_n \to \bZ/2$ the homomorphism that sends $R$ to $0$ and $S$ to $1$.

\begin{definition}
Let $C$ be a $\bZ/n$-torsor and let $\tau$ be a $\bZ/2$-torsor.  An \emph{unwinding} of $(C,\tau)$ is tuple $(\tilde{C},\sigma,\rho)$ where $\tilde{C}$ is a $G_n$-torsor, $\rho:\tilde{C} \to C$ is a map equivariant with respect to the homomorphism $\rho:G_n \to \bZ/n$, and $\sigma:\tilde{C} \to \tau$ is map equivariant with respect to the homomorphism $\sigma:G_n \to \bZ/2$.  We refer to the diagram $C \stackrel{\rho}{\leftarrow} \tilde{C} \stackrel{\sigma}{\rightarrow} \tau$ as a \emph{$G_n$-unwinding}.
\end{definition}

\begin{remark}
Just as we visualize a $\bZ/n$-torsor as a cyclically ordered set, we can visualize an unwinding $C \leftarrow \tilde{C} \to \tau$ as two rows of elements (indexed by elements of the $\bZ/2$-torsor $\tau$), a ``shift-to-the-right'' operation $R$ on each row, and an operation $S$ that goes back and forth between rows.  For instance, here is a picture of a $3$-unwinding:
$$\xymatrix{\\ \dots}\quad\xymatrix{
\dots & \bullet \ar[rrrd]^{S} \ar[rr]^{R} & {} & \bullet \ar[rrrd]^{S} \ar[rr]^{R} & {} & \bullet \ar[rrrd]^{S} \ar[rr]^{R} & {} & \bullet & {} \dots \\
\bullet \ar[rrru]^{S} \ar[rr]^{R} & {} & \bullet \ar[rrru]^{S} \ar[rr]^{R} & {} & \bullet \ar[rrru]^{S} \ar[rr]^{R} & {} & \bullet \ar[rr]^{R}& {} & \bullet}
\quad\xymatrix{ \\ \dots}$$
\end{remark}

\begin{remark}
Let $C$ be a cyclically ordered set with $n$ elements, let $\tau$ be a $\bZ/2$-torsor, and let $\tau \leftarrow \tilde{C} \rightarrow C$ be an unwinding of $(C,\tau)$.  The action of $S:\tilde{C} \to \tilde{C}$ gives the projection $\tilde{C} \to C$ the structure of a $\bZ$-torsor.  We can describe the action $R:\tilde{C} \to \tilde{C}$ as a ``connection'' on this $\bZ$-torsor: if we regard a minimal pair $(x,y)$ as a ``path'' in $C$, then $R:\tilde{C}_x \to \tilde{C}_y$ can be understood as the ``monodromy'' along this path.  Similarly $R^k:\tilde{C}_{x_0} \to \tilde{C}_{x_k}$ is the monodromy along the minimal chain $(x_0,\ldots,x_k)$.  The relation $R^n = S^2$ can be understood to mean that the monodromy over a minimal loop in $C$ is given by $S^2$.
\end{remark}

\begin{remark}
\label{rem:subunwind}
Let $\tau$ be a $\bZ/2$ torsor, let $C$ be a cyclically ordered set, and let $C' \subset C$ be a subset with the induced cyclic order.  If $\tau \leftarrow \tilde{C} \stackrel{\rho}{\to} C$ is an unwinding of $(C,\tau)$ then $\tau \leftarrow \rho^{-1}(C') \to C'$ is an unwinding of $(C',\tau)$, which we will call the ``induced unwinding'' of $(C',\tau)$.  If $C'$ has $n$ elements, then we can describe the $G_n$-action on $C'$ using the language of the previous remark:
\begin{itemize}
\item $S$ acts just as it does on $C$,
\item If $(x,y)$ is a minimal pair in $C'$ then $R:\rho^{-1}(C')_x \to \rho^{-1}(C'_y)$ is the monodromy in $\tilde{C} \to C$ along the unique minimal chain $(x_0,\ldots,x_k)$ in $C$ with $x_0 = x$ and $x_k = y$.
\end{itemize}
\end{remark}

\subsubsection{$\bZ$-graded ribbon graphs}

\begin{definition}
\label{def:zgrg}
A \emph{$\bZ$-graded ribbon graph} consists of the following data:
\begin{itemize}
\item A $\bZ/2$-graded ribbon graph $(X,\tau)$.
\item A $\bZ$-torsor $\tilde{\tau}_e$ for every edge $e \subset X$, together with a map $\tilde{\tau}_e \to \tau_e$ that is equivariant for the homomorphism $\bZ \to \bZ/2$.
\item An unwinding $\tilde{C}_v \stackrel{q}{\to} C_v$ of $(C_v,\tau)$ for every vertex $v$ of $X$
\item For each half edge $\epsilon$ incident with vertex $v$ and edge $e$, an isomorphism $\theta_\epsilon$ between the $\bZ$-torsors $q^{-1}(\epsilon)$ and $\tilde{\tau}_e$.
\end{itemize}
\end{definition}

\begin{example}
\label{ex:chordal}
If $(X,Z)$ is a chordal ribbon graph, we may endow $X$ with a grading in the following way.
First, since $Z$ is bivalent, the underlying topological space is a one-manifold, and we may choose an orientation (the grading will not depend on this choice, up to isomorphism) which will identify up to homotopy a neighborhood $Star_Z(b)\subset Z$ of each $b\in V_Z$ with a neighborhood of $0\in \bR$.  Further, $Star_X(b)\subset X$ can be identified with a conical Lagrangian in $T^*\bR$ as follows:  call $\sf E$ the half edge of $Star_Z(b)$ identified with $\bR_+$ and $\sf W$ the half edge corresponding to $\bR_-$, then label a half edge $e_b$ of $Star_X(b)$ by $\sf N$
if $({\sf E},e,{\sf W})$ is in the cyclic order of the set $C_b$ of half edges at $b,$ and by $\sf S$ if
$({\sf W},e,{\sf E})$ is in the cyclic order of $C_b$ (by Definition \ref{def:chordal} this defines a unique labeling of $Star_X(b)$).

Now let $\tau$ be the trivial $\bZ_2$-torsor on $X$ and for each $e\in E_X$ let $\tilde{\tau}_e$ be
the trivial $\bZ$-torsor.  Let $D = \{\sf{E,N,W,S}\}$ be the cyclically ordered set corresponding
to the standard compass, and let $\tilde{D}\cong G_4\rightarrow D \cong \bZ/4$ be the unwinding defined by setting $E\leftrightarrow 0\in \bZ/4$.  Since each $C_b\subset D,$ each $C_b$ inherits the induced unwinding $\tilde{C}_b\stackrel{q}{\rightarrow} C_b$ by Remark \ref{rem:subunwind}.
Finally, the compatibility isomorphisms $\theta$ of Definition \ref{def:zgrg} for the unwindings $q$
relate $n\in \bZ$ to the corresponding elements of one of $q^{-1}({\sf E}) = \{S^n\},$
$q^{-1}({\sf N}) = \{RS^n\},$ $q^{-1}({\sf W}) = \{R^2S^n\},$ $q^{-1}({\sf S})= \{R^3S^n\}.$
\end{example}

\begin{remark}
Let $(X,\tau,\tilde{\tau}_e,\{\tilde{C}_v\})$ be a $\bZ$-graded ribbon graph and let $U \subset X$ be an open subgraph.  Then we give $U$ an induced $\bZ$-grading by restricting the data $\tau$, $\{\tilde{\tau}_e\}$, $\tilde{C}_v$ to $U$.

On the other hand let $Y \subset X$ be a closed subgraph, and suppose that $Y$ has no vertices of degree $1$, so that $Y$ is a ribbon graph.  Endow $Y$ with a $\bZ/2$-torsor $\tau_Y$ by restricting $\tau$ to $Y$ and each edge $e \subset Y$ with the $\bZ$-torsor $\tilde{\tau}_e$, and for each vertex $v$ of $Y$ give the set of half-edges of $Y$ incident with $v$ the unwinding induced from the unwinding of the half-edges of $X$ incident with $v$, in the sense of Remark \ref{rem:subunwind}.  Thus $Y$ inherits a $\bZ$-grading.
\end{remark}

\begin{remark}
Let $(X,\tau,\tilde{\tau}_e,\{\tilde{C}_v\})$ be a $\bZ$-graded ribbon graph.  A noncompact boundary component $e_1,\ldots,e_r$ of $X$ induces a map from $\tilde{\tau}_{e_1} \to \tilde{\tau}_{e_2}$ by applying the operator $R$ once at each vertex.  We refer to this map as the ``monodromy along the boundary component.''
\end{remark}

We have the following $\bZ$-graded analog of proposition \ref{prop:jointree}.

\begin{definition}
Let $(X,\tau,\{\tilde{C}_v\}_{v \in V_X})$ be a $\bZ$-graded ribbon tree.  Since $X$ is contractible we may abuse notation and let $\tau$ also denote the $\bZ/2$-torsor of sections of $\tau \to X$.  We endow the leaves $L$ of $X$ with the cyclic order of \ref{prop:jointree}, and define an unwinding $(L,\tau)$ as follows:
\begin{itemize}
\item The total space $\tilde{L}$ is $\coprod_{e \in L} \tilde{\tau}_e$.  The projection $\tilde{L} \to L$ is the tautological one, and the projection $\tilde{L} \to \tau$ is induced by the projections $\tilde{\tau}_e \to \tau$ that are part of the data of a grading on $X$.
\item The action of $S$ is given by the $\bZ$-torsor structure on $\tau_e$.
\item The action of $R$ is given by monodromy along the noncompact boundary components.
\end{itemize}
\end{definition}

We can now define a simple contraction of $\bZ$-graded ribbon trees.

\begin{definition}
Let $X$ and $Y$ be $\bZ$-graded ribbon graphs.  A \emph{simple contraction} $f:X \to Y$ is given by the following data:
\begin{enumerate}
\item A simple contraction of the underlying ribbon graphs.
\item An isomorphism between the $\bZ/2$-torsors $\tau_X$ and $f^* \tau_Y$
\item For each edge $e$ of $X$ that is not contracted, an isomorphism between $\tilde{\tau}_e$ and $\tilde{\tau}_{f(e)}$
\end{enumerate}
For each vertex $v \in Y$, the isomorphism (2) and (3) are required to be compatible with the unwindings of the leaves of $\mathrm{star}(v)$ and those of $f^{-1}(\mathrm{star}(v))$
\end{definition}

Now we can define $\bZ$-graded Fukaya formalisms.  Let $X$ and $Y$ be $\bZ$-graded ribbon graphs.  A \emph{partial contraction} from $X$ to $Y$ is a diagram
$$X \stackrel{j}{\hookleftarrow} U \stackrel{p}{\to} Y$$
where $j$ is an open inclusion and $p$ is a simple contraction.  We write $\RGpc^\bZ$ for the category of ribbon graphs and partial contractions, and $\RTpc^{\bZ}$ for the full subcategory of $\bZ$-graded ribbon trees and partial contractions.  We define the subcategories $(\RGo^\bZ)^{\op}$ , $\RGc^\bZ$, $(\RTo^\bZ)^{\op}$, $\RTc^{\bZ}$ of $\RGpc^\bZ$ and $\RTpc^\bZ$ as before.  

We endow $(\RGpc^\bZ)^{\op}$ with the coarsest Grothendieck topology compatible with the forgetful functor $\RGpc^\bZ \to \RGpc$---it has $(\RTpc^\bZ)^{\op}$ as a basis.  The topology on $(\RGpc^{\bZ})^{\op}$ determines a topology on $(\RGo^\bZ)_{/X}$ for each graded ribbon graph $X$.

\begin{definition}
A \emph{$\bZ$-graded Fukaya formalism} is a functor from $\RGpc^\bZ$ to the $\infty$-category $\dgCat$, satisfying the sheaf condition.
\end{definition}

\begin{proposition}
Let $\Fuk:\RGpc^\bZ \to \dgCat$ be a Fukaya formalism.  The following hold:
\begin{enumerate}
\item For each graded ribbon graph $X$, $\Fuk\vert_{(\RGo^{\bZ})^{\op}_{/X}}$ is a sheaf.
\item Whenever $X \to Y$ is a simple contraction, $\Fuk(X) \to \Fuk(Y)$ is an equivalence.
\end{enumerate}
\end{proposition}

% 
% \subsection{Sketch of a construction of a $\bZ$-graded Fukaya formalism}
% 
% {\color{red} In this section, we sketch a construction of a $\bZ$-graded Fukaya formalism.  Some or most or all of the credit for this section goes to others (Nadler, Kontsevich): it has been recognized by many people for some time that the Fukaya category of an n-valent vertex is an $A_{n-1}$-quiver.  The fact that this formalism agrees with $\CPM$ has the following interesting consequence: the mapping class group of a marked genus 1 (resp. genus 0) surface acts on the category of perfect complexes on a balloon ring (resp. balloon chain).}

%\subsection{The construction of Fukaya formalisms}

%In this section we show that a Fukaya formalism is essentially determined by its behavior on ribbon trees $X$ with the property that each vertex of $X$ has degree $\leq 3$.

%\begin{definition}
%For each integer $n \geq 0$, let $\RTpc^{\leq n} \subset \RTpc$ denote the full subcategory of ribbon graphs with the property that each vertex has degree $\leq n$.  We endow $\RTpc^{\leq n}$ with the topology induced by that from $\RTpc$.  
%\end{definition}

%\begin{proposition}
%Let $\cC$ be an $\infty$-category with limits and colimits.  For $n \geq 3$, the restriction functor $\Shv(\RTpc,\cC) \to \Shv(\RTpc^{\leq n},\cC)$ is fully faithful.
%\end{proposition}

%\begin{proof}
%By the $\infty$-categorical Yoneda lemma, we may reduce to the case $\cC = \cS$ is the $\infty$-category of spaces.  
%\end{proof}

\end{document}